\documentclass[reqno,12pt]{amsart}

\usepackage{amsfonts, amssymb, verbatim}
\usepackage{amssymb,epsfig,amscd, verbatim}
\usepackage{amsthm}
\usepackage{amsmath}
\usepackage{accents}
\usepackage{graphicx}
\usepackage{latexsym}
\usepackage{tikz-cd}

\usepackage[all]{xy}

\usepackage{array}
\usepackage{float}
\usepackage{wrapfig}

\allowdisplaybreaks

\newtheorem{theorem}{Theorem}

\newtheorem{proposition}[theorem]{Proposition}

\theoremstyle{definition}
\newtheorem{definition}[theorem]{Definition}

\numberwithin{equation}{section} \numberwithin{theorem}{section}

\theoremstyle{remark}

\def\l{\lambda}

\newcommand\lbb[1]{\label{#1}}

\def\tt{\otimes}                               
\def\<{\langle}
\def\>{\rangle}

\def\d{\partial}

\def \tilde{\widetilde}
\def \<{\langle}
\def \>{\rangle}
\def \p{\partial}
\def \c{{*c}}

\def \lbb{\label}

\def\dsum{\displaystyle\sum}

\newcommand{\CC}{\mathbb{C}}

\newcommand{\ZZ}{\mathbb{Z}}

\newcommand{\fg}{\mathfrak{g}}

\def\al{\alpha}                         
\def\be{\beta}

\def\de{\delta}
\def\De{\Delta}

\def\la{\lambda}
\def\La{\Lambda}

\def\om{\omega}

\def \tilde{\widetilde}
\def \la{\lambda}
\def \p{\partial}
\def \pa{\partial}
\def \cp{\mathbb C[\partial]}
\def \l0{L_{\geq 0}}

\begin{document}

\title[Classification of finite simple differential Lie and Jordan
(super)coalgebras]{Classification of simple differential Lie and Jordan
(super)coalgebras of finite rank}

\author [Carina Boyallian and Jos\'e I. Liberati]{Carina Boyallian$^*$ and Jos\'e I. Liberati$^{**}$}
\thanks {\textit{$^{*}$Famaf - Ciem (CONICET), Medina Allende y
Haya de la Torre, Ciudad Universitaria, (5000) C\'ordoba \indent - \indent
 Argentina. Email:
cboyallian@unc.edu.ar}\newline
\indent \textit{$^{**}$Ciem - CONICET, Medina Allende y
Haya de la Torre, Ciudad Universitaria, (5000) C\'ordoba \ - \indent
Argentina. Email:
joseliberati@gmail.com}}
\address{\textit{Famaf - Ciem (CONICET), Medina Allende y
Haya de la Torre, Ciudad Universitaria, (5000) C\'ordoba - \indent
 Argentina. e-mail:
cboyallian@unc.edu.ar}\newline
\indent {\textit{Ciem - CONICET, Medina Allende y
Haya de la Torre, Ciudad Universitaria, (5000) C\'ordoba -
Argentina.  \indent e-mail: joseliberati@gmail.com}}}

\date{version 10 Dic, 2024}


\maketitle

\begin{abstract}
We classify simple differential Lie and Jordan (super)coalgebras of finite rank. In particular, we provide an explicit description of the Lie supercoalgebras associated with the operator product expansion (OPE) of the
$
N=2,3,4$ superconformal Lie algebras and the exceptional Lie conformal superalgebra CK$_6$.
\end{abstract}


\section{Introduction}\lbb{intro}

Finite simple Lie conformal algebras were classified in \cite{DK}, and all their finite irreducible representations were constructed in \cite{CK}. According to \cite{DK}, any finite simple Lie conformal algebra is isomorphic either to the current Lie conformal algebra \textit{Cur} $\fg$, where $\fg$ is a simple finite-dimensional Lie algebra, or to the Virasoro conformal algebra.

However, the list of finite simple Lie conformal superalgebras is much richer, mainly due to the existence of several series of super extensions of the Virasoro conformal algebra. The complete classification of finite simple Lie conformal superalgebras was obtained in \cite{FK}. The list consists of current Lie conformal superalgebras \textit{Cur} $\fg$, where $\fg$ is a simple finite-dimensional Lie superalgebra; four series of "Virasoro-like" Lie conformal superalgebras $W_n$ ($n \geq 0$), $S_{n,b}$ and $\tilde{S}_n$ ($n \geq 2, b \in \CC$), $K_n$ ($n \geq 0$); and the exceptional Lie conformal superalgebra $CK_6$.

In \cite{L}, one of the authors introduced conformal Lie coalgebras and conformal Lie bialgebras.
 The notion of  conformal Lie coalgebra is referred to as a differential Lie coalgebra in this work. Using the results in \cite{L}, we prove a correspondence between  the category of finite free  differential Lie (super)coalgebras and  the category of finite free  Lie conformal (super)algebras (Section 2). Since simple finite Lie conformal (super)algebras correspond to simple finite differential Lie (super)coalgebras, the classification of finite simple Lie conformal algebras (resp. superalgebras) obtained by D'Andrea and Kac \cite{DK} (resp. Fattori and Kac \cite{FK}) yields the classification of simple differential Lie (super)coalgebras of finite rank (Section 3).

Since Lie conformal algebras encode the OPE, we are thereby obtaining the Lie coalgebra associated with the OPE. In particular, we present the Lie supercoalgebras associated with the OPE of the $N=2, 3, 4$ superconformal Lie algebras and $CK_6$ (Sections 4.6 and 4.5).

One of the main results of the present paper is the explicit description of the simple differential Lie (super)coalgebras of finite rank (Section 4).

Similarly, the simple finite Jordan conformal superalgebras were classified by Kac and Retakh \cite{KR}. Using the correspondence of categories that also holds for Jordan conformal superalgebras, we obtain the classification of simple finite differential Jordan supercoalgebras, along with their explicit description (Section 5).

\section{Lie conformal superalgebras and differential Lie supercoalgebras}\lbb{Lie}

In this section we present the definitions of Lie conformal superalgebras and differential Lie supercoalgebra, and we establish a correspondence of categories between the finite free objects in each category.

\begin{definition} A {\it   Lie conformal superalgebra} $R$ is  a left
$\ZZ/2\ZZ$-graded $\cp$-module endowed with a $\CC$-linear map,
\begin{displaymath}
R\otimes R  \longrightarrow \CC[\la]\otimes R, \qquad a\otimes b
\mapsto [a_\la b]
\end{displaymath}
called the $\la$-bracket satisfying the following axioms
$(a,\, b,\, c\in R)$,

\

\noindent Conformal sesquilinearity $ \qquad  [\pa a_\la b]=-\la
[a_\la b],\qquad [a_\la \pa b]=(\la+\pa) [a_\la b]$,

\vskip .3cm

\noindent Skew-symmetry $\ \qquad\qquad\qquad [a_\la
b]=-(-1)^{p(a)p(b)}[b_{-\la-\pa} \ a]$,

\vskip .3cm

\noindent Jacobi identity $\quad\qquad\qquad\qquad [a_\la [b_\mu
c]]=[[a_\la b]_{\la+\mu} c] + (-1)^{p(a)p(b)}[b_\mu [a_\la c]]. $

\vskip .5cm

Here and further $p(a)\in \ZZ/2\ZZ$ is the parity of $a$.
\end{definition}

A Lie conformal superalgebra is called {\it finite} if it has finite
rank as a $\CC[\pa]$-module, and it is called {\it free} if it is free 
 as a $\CC[\pa]$-module. The notions of homomorphism, ideal
and subalgebras of a Lie conformal superalgebra are defined in the
usual way. A Lie conformal superalgebra $R$ is {\it simple} if $[R_\la
R]\neq 0$ and contains no ideals except for zero and itself.

We define the
 {\it conformal dual} of a $\cp$-module $U$ as
\begin{equation}\label{eq:dual}
U^{\c}=\{f:U\to \CC[\la]\ | \ \CC\hbox{-linear and } f_{\la}(\pa
b)= \la f_{\la}(b)\}.
\end{equation}
It is a $\cp$-module with
\begin{equation}\label{eq:p}
(\pa f)_{\la}:= -\la f_{\la}.
\end{equation}
Given a homomorphism of $\cp$-modules $T:M\to N$, we define the transpose homomorphism $T^*:N^{\c} \to M^{\c}$ given by
\begin{equation*}
  [\, T^*(f)\, ]_\la (m)= f_\la (\,T(m)).
\end{equation*}

\vskip .1cm

We also define the tensor
 product $U\otimes V$  of $\cp$-modules as the ordinary tensor product with
 $\cp$-module structure $(u\in U, v\in V)$:
$$
\pa(u\otimes v)\,=\, \pa u\otimes v + u\otimes \pa v.
$$

\begin{definition} A {\it differential Lie supercoalgebra} $(C,\delta)$ is a $\ZZ/2\ZZ$-graded $\CC[\partial]$-module $C$ endowed with a $\CC[\partial]$-homomorphism
\begin{displaymath}
\delta:C\to C\tt C
\end{displaymath}
such that
\begin{equation*}
  \mathrm{Im}(\de)\subseteq \hbox{span}\{ a\tt b - (-1)^{p(a) p(b)} b\tt a\, :\, a,b\in C\},
\end{equation*}
\noindent and
\begin{displaymath}
(I\otimes \delta) \delta - (\tau \tt I) (I\otimes \delta) \delta=
(\delta\otimes I) \delta,
\end{displaymath}
\vskip .3cm

\noindent
where $\tau(a\otimes b )=(-1)^{p(a) p(b)}\, b\otimes a$.
\end{definition}

\noindent That is, the standard definition of a Lie supercoalgebra, with a compatible
$\CC[\partial]$-structure. A differential Lie supercoalgebra is called {\it finite} if it has finite
rank as a $\CC[\pa]$-module, and it is called {\it free} if it is free 
 as a $\CC[\pa]$-module. The notions of homomorphism 
and subsupercoalgebras of a differential Lie supercoalgebra are defined in the
usual way. For example, a $\cp$-submodule $C_0$ of $C$ is a {\it subsupercoalgebra}, if $\de(C_0)\subseteq C_0\tt C_0$. A differential Lie supercoalgebra $(C,\delta)$ is {\it simple} if $\delta\neq 0$ and contains no subsupercoalgebras except for zero and itself.

We shall need the following result, that follows from Proposition 2.7 in \cite{BKLR}:

\begin{proposition} \label{varphi-}
  \cite{BKLR} Let $V$ be a finite free $\cp$-module, we define  $\varphi:V\to (V^{\c})^\c$ by
\begin{equation*}
  [\,\varphi(v)\,]_\mu (f):=f_{-\mu} (v),
\end{equation*}
for any $v\in V$, $f\in V^\c$. Then, $\varphi$  is an isomorphism of $\cp$-modules.
\end{proposition}

The following results were obtained in \cite{L}, except for $(c)$ and $(d)$.

\begin{proposition}\label{prop:phi} {\rm (Proposition 2.12, \cite{L})} Let $R$ be a finite free $\cp$-module. Let $\Phi: R^{*c}\otimes R^{*c}\to \CC[\mu]\otimes(R\otimes R)^{*c}$ given by
\begin{equation}\label{eq:0}
\left[\Phi_\mu (f\otimes g)\right]_\lambda (r\otimes r')=
f_\mu(r)\  g_{\lambda -\mu} (r').
\end{equation}
Then we have:

(a) $\Phi_\mu(\p f\otimes g) =-\mu \Phi_\mu (f\otimes g)$ and
$\Phi_\mu( f\otimes \p g) =(\p +\mu) \Phi_\mu (f\otimes g)$

(b) $\Phi$ is a homomorphism of $\CC[\p]$-modules.

(c) $\Phi$ is injective.

(d) If $f:R\to S$ is a homomorphism of $\CC[\p]$-modules, then  $\Phi_{\mu}\circ (f^*\tt f^*)  = (f\tt f)^* \circ \Phi_{\mu}$.
\end{proposition}
\begin{proof} (c) Let $0\neq f_i,g_i\in R^\c$ be such that $0=\Phi_\mu (\sum_{i=1}^n f_i\tt g_i)$. We may suppose that $\{f_i\}$ are $\CC$-linearly independent.
Since $g_i\neq 0$ for all $i$, there exists $b\in R$ such that $0 \neq (g_j)_y(b)$ for some $j$. If for each $i$, $(g_i)_y(b)=\sum_k \alpha^i_k y^k$, with $\alpha^i_k\in \mathbb{C}$, then we have (for all $a\in R$)
\begin{align*}
  0 &  = [\Phi_\mu (\textstyle{\sum_i} (f_i\tt g_i))]_{\la+\mu}(a\tt b)= \textstyle{\sum_i} (f_i)_\mu(a) \ (g_i)_\la (b)=\textstyle{\sum_i} (f_i)_\mu(a) (\textstyle{\sum_k} \alpha^i_k \la^k) \\
& =\textstyle{\sum_k} \big(\textstyle{\sum_i}(f_i)_\mu(a)  \alpha^i_k \big)\la^k.
\end{align*}
Therefore $0=\textstyle{\sum_i}(\alpha^i_k \,f_i)_\mu(a) $ for all  $a\in R$ and all $k$. Then $0=\textstyle{\sum_i} \alpha^i_k \,f_i$, obtaining that $\al_k^i=0$ for all $i,k$. In particular $(g_j)_y(b)=0$, which is a contradiction.

\vskip .1cm

\noindent (d) For any $h,g\in S^\c$, $r,s\in R$, we have
\begin{align*}
  \big([\Phi_{\mu}\circ (f^*\tt f^*) ] (h\tt g)\big)_\la (r\tt s)
& = [\Phi_{\mu} (f^*(h)\tt f^*(g)) ]_\la (r\tt s)
= f^*(h)_\mu (r) \ f^*(g)_{\la-\mu}(s)\\
   & = h_\mu (f(r)) \ g_{\la-\mu}(f(s))
= [\Phi_{\mu} (h\tt g) ]_\la ((f\tt f)(r\tt s))\\
& = \big([(f\tt f)^* \circ \Phi_{\mu} ] (h\tt g)\big)_\la (r\tt s),
\end{align*}
finishing the proof.
\end{proof}

%
%
%
%
The next result was obtained in \cite{L} for Lie conformal algebras. The proof for superalgebras is similar.

\begin{theorem}
\label{prop:dual}  {\rm (Proposition 2.13, \cite{L})}
(a) Let $(C, \delta)$ be a differential Lie (super)coalgebra, then $C^{*c}$ is a Lie conformal (super)algebra with the following bracket $(f,g\in C^\c)$:
\begin{equation}\label{eq:2}
([f_\mu g])_\lambda (r)= \sum \, f_\mu(r_{(1)}) \ g_{\lambda -\mu}
(r_{(2)}) = [\Phi_\mu (f\otimes g)]_{\la} (\de(r)),
\end{equation}
where $\delta(r)=\sum \, r_{(1)}\otimes r_{(2)}$.

\vskip .2cm

\noindent (b) Let $(R, [\  _\la \  ])$ be a Lie conformal (super)algebra free 
  of finite rank, that is $R=\oplus_{i=1}^n \cp a_i$, then
$R^{*c}=\oplus_{i=1}^n \cp a_i^*$, where $\{a_i^*\}$ is a
dual $\cp$-basis in the sense that $(a_i^*)_\la(a_j)=\tilde\delta_{ij}$ (where $\tilde\delta_{i,j}$ is the Kronecker delta),
is a differential Lie (super)coalgebra with the following co-product:
\begin{equation} \label{eq:De}
\de (f)= \sum_{i,j} \, f_\mu([{a_i}_\la a_j]) \ (a_i^*\otimes a_j^*)|_{
\la=\p\otimes 1,\ \mu= -\p\otimes 1 - 1\otimes \p }
\end{equation}
More precisely, if
\begin{displaymath}
[{a_i}_\la a_j]=\sum_k P^{ij}_k (\la, \p) a_k,
\end{displaymath}
where $P^{ij}_k$ are some polynomials in $\la$ and $\p$,  then the
co-product is
\begin{displaymath}
\de(a_k^*)=\sum_{i,j} Q^{ij}_k(\p\otimes 1, 1\otimes \p)\ a_i^*\otimes
a_j^*.
\end{displaymath}
where $Q^{ij}_k(x,y)=P^{ij}_k(x, -x-y)$.
\end{theorem}

\

Now, using   Theorem \ref{prop:dual}, we have:

\begin{theorem}\label{equival}
  The construction in Theorem \ref{prop:dual} (a), defines a contravariant  functor $F$ from  the category of finite free differential Lie (super)coalgebras  
to the category of finite free Lie conformal (super)algebras. Conversely, the construction in Theorem \ref{prop:dual} (b), defines a contravariant  functor $G$ from  the category of finite free Lie conformal (super)algebras 
to the category of finite free differential Lie (super)coalgebras. Moreover, we have that $G\circ F\cong 1_\mathcal{C}$, where $\mathcal{C}$ is the category of finite free differential Lie (super)coalgebras, and we also have that $F\circ G\cong 1_\mathcal{L}$, where $\mathcal{L}$ is the category of finite free Lie conformal (super)algebras.
\end{theorem}

\begin{proof}
   Given a finite free differential Lie (super)coalgebra $(C,\delta)$, we define    $F(C,\delta):=(C^\c,[\ \,_\mu\ \, ])$, where $[\ \,_\mu\ \, ]:=(\mathrm{id}\tt\delta^* )\circ \Phi_\mu$. By Theorem \ref{prop:dual} (a), we have that $F(C,\delta)$ is a finite free Lie conformal (super)algebra. And, if $h:C_1\to C_2$ is a homomorphism of differential Lie (super)coalgebras, we define $F(h):=h^*:C_2^\c\to C_1^\c$. Observe that $h^*$ is a homomorphism of Lie conformal (super)algebras since ($f,g\in C_2^\c$)
\begin{align*}
  [h^*(f)_\mu h^*(g)]_\la (r)
& = \sum \, (h^*(f))_\mu(r_{(1)}) \ (h^*(g))_{\lambda -\mu}
(r_{(2)}) =
\sum \, f_\mu(h(r_{(1)})) \ g_{\lambda -\mu}
(h(r_{(2)})) \\
& = [\Phi_\mu (f\otimes g)]_{\la} ((h\tt h)\de(r))
= [\Phi_\mu (f\otimes g)]_{\la} (\de(h(r)))
= h^*([f_\mu g])_\la (r).
\end{align*}
It easy to see that $F$ is a contravariant functor.

Let $(R, [\ \,_\mu\ \, ])$ be a finite free Lie conformal (super)algebra. We define $\pi_\mu : R^\c \to \mathbb{C}[[\mu]] \tt (R\tt R)^\c $ by
 ($f\in R^\c$ and $a,b\in R$)
\begin{equation}\label{piii}
  [\, \pi_\mu (f)\,]_\la (a \tt b)=f_\la ([a_\mu b])  \in  \mathbb{C}[\mu,\la].
\end{equation}

\vskip .1cm

\noindent Observe that there exist a unique map $\de:R^\c \to R^\c \tt R^\c$ such that $\Phi_{-\mu} \circ \de =\pi_\mu$, that is, satisfying the following commutative diagram

\[
\begin{tikzcd}
R^\c \arrow[r, "\pi_\mu"] \arrow[dr, dashed, "\delta"'] & \mathbb{C}[\mu]\otimes (R \otimes R)^\c \\
& R^\c \otimes R^\c  \arrow[u, hook, "\Phi_{-\mu}"'] 
\end{tikzcd}
\]

\vskip .2cm
 
\noindent More precisely, we shall prove that $\de$ defined in Theorem \ref{prop:dual} (b) satisfies this relation, and using that $\Phi_\mu$ is injective (see Proposition \ref{prop:phi}), we obtain that $\de$ is unique with this property. Using the notation in Theorem \ref{prop:dual} (b), we have
\begin{equation*}
  [\pi_\mu(a_k^*)]_\la (a_i\tt a_j)=(a_k^*)_\la ([{a_i}_\mu a_j ]) =(a_k^*)_\la \big(\sum_l \, P_l^{i j}(\mu,\p)\, a_l\big)=P_k^{i j}(\mu,\la),
\end{equation*}
and
\begin{align*}
   [\,\Phi_{-\mu} ( \de (a_k^*))\, ]_\la (a_i\tt a_j) & =
\big[\,\Phi_{-\mu} \big(\sum_{l,m} Q^{lm}_k(\p\otimes 1, 1\otimes \p)\ (a_l^*\otimes
a_m^*)\big)\, \big]_\la (a_i\tt a_j)  \\
& = \sum_{l,m} Q^{lm}_k(\mu, -\la -\mu)\ (a_l^*)_{-\mu}(a_i)
\ (a_m^*)_{\la+\mu}(a_j)=Q^{ij}_k(\mu, -\la -\mu),
\end{align*}
therefore $\de$ satisfies $\Phi_{-\mu} \circ \de =\pi_\mu$. Using this, we define $G(R, [\ \,_\mu\ \, ]):=(R^\c , \de)$, and $(R^\c , \de)$ is a finite free differential Lie (super)coalgebra by Theorem \ref{prop:dual} (b). If $f:R_1\to R_2$ is a homomorphism of Lie conformal (super)algebras, we define $G(f):=f^*:R_2^\c\to R_1^\c$. Now we have to prove that  $f^*$ is a homomorphism of differential Lie (super)coalgebras, that is
\begin{equation}\label{fff}
\de_1 \circ f^*=(f^*\tt f^*)\circ \de_2,
\end{equation}
where $\de_i$ is the differential Lie (super)coalgebra structure associated to $R_i^\c$.  Observe that
\begin{equation*}
\Phi_{-\mu}\circ \de_1 \circ f^*=\pi_\mu^1 \circ f^*,
\end{equation*}
and using Proposition \ref{prop:phi} (d), we have
\begin{equation*} 
\Phi_{-\mu}\circ (f^*\tt f^*) \circ  \de_2= (f\tt f)^* \circ \Phi_{-\mu}\circ \de_2
=(f\tt f)^*\circ \pi_\mu^2,
\end{equation*}
where $\pi_\mu^i$ is  associated to the Lie conformal (super)algebra structure in $R_i$. Therefore, since $\Phi_\mu$ is injective, we have that (\ref{fff})
 is equivalent to prove that
$\pi_\mu^1 \circ f^*=(f\tt f)^*\circ \pi_\mu^2$, and this follows by
 ($h\in R_2^\c$, $r,s\in R_1$)
\begin{align*}
 [\pi_\mu^1 ( f^*(h))]_\la (r\tt s) & =
(f^*(h))_\la ([r_\mu s])=h_\la(f([r_\mu s])= h_\la([f(r)_\mu f(s)]) \\
& =
[\, \pi_\mu^2(h)\, ]_\la((f\tt f)(r\tt s))
= [(f\tt f)^*([\, \pi_\mu^2(h)\, ])]_\la(r\tt s).
\end{align*}
It easy to see that $G$ is a contravariant functor.

Recall Proposition \ref{varphi-}: for any  finite free $\cp$-module $V$, the map $\varphi:V\to (V^{\c})^\c$ defined by ($v\in V$, $f\in V^\c$) 
\begin{equation*}
  [\,\varphi(v)\,]_\mu (f):=f_{-\mu} (v),
\end{equation*}
is an isomorphism of $\cp$-modules. Using this map, 
if $(C,\de)$ is a  finite free differential Lie (super)coalgebra, and $G(F(C,\de))=((C^{\c})^\c, \tilde{\de})$, then we have to prove that $\varphi:(C,\de)\to ((C^{\c})^\c, \tilde{\de})$ is an isomorphism of differential Lie (super)coalgebras, that is
$(\varphi\tt \varphi)\circ \de=\tilde\de\circ \varphi$. But, using that $\Phi_\mu$ is injective, this is equivalent to prove that
\begin{equation}\label{isoo}
  \Phi_{-\mu}\circ (\varphi\tt \varphi)\circ \de=\Phi_{-\mu}\circ \tilde\de\circ \varphi = \pi_\mu \circ \varphi,
\end{equation}
where $\pi_\mu$ is the map (\ref{piii}) associated to the  Lie conformal (super)algebra structure in  $F(C,\delta)=(C^\c,[\ \,_\mu\ \, ])$.  Since the RHS in (\ref{isoo}) is given by ($c\in C$, $f,g\in C^\c$)
\begin{equation*}
  ([\pi_\mu \circ \varphi](c))_\la (f\tt g)=(\varphi(c))_\la ([f_\mu g])
=([f_\mu g])_{-\la}(c),
\end{equation*}
and the LHS in (\ref{isoo}) is given by
\begin{align*}
  ([\Phi_{-\mu}\circ (\varphi\tt \varphi)\circ \de](c))_\la (f\tt g)
& = \big[\Phi_{-\mu}  \big(\textstyle{\sum} \, \varphi(c_{(1)})\tt \varphi(c_{(2)})\big)\big]_\la (f\tt g)\\
& = \textstyle{\sum} \,  \ [\varphi(c_{(1)})]_{-\mu}(f)\  [\varphi(c_{(2)})]_{\la+\mu}(g)
=   \textstyle{\sum} \,  \ f_{\mu}(c_{(1)}) \    g_{-\la-\mu}(c_{(2)})\\
& =([f_\mu g])_{-\la}(c),
\end{align*}
we proved that $\varphi$ is an isomorphism of differential Lie (super)coalgebra. It is clear that if $h:C_1\to C_2$ is a homomorphism of differential Lie (super)coalgebras, then we have  $\varphi\circ h= G(F(h))\circ \varphi$. Hence, we obtained that $G\circ F\cong 1_\mathcal{C}$, where $\mathcal{C}$ is the category of finite free differential Lie (super)coalgebras.

Conversely, if $(R, [\ \,_\mu\ \, ])$ be a finite free  Lie conformal (super)algebra, then  $G(R, [\ \,_\mu\ \, ])=(R^\c , \de)$ with $\Phi_{-\mu} \circ \de =\pi_\mu$, and $F(G(R, [\ \,_\mu\ \, ]))=((R^\c)^\c , [\ \,_\mu\ \, ]\, \tilde{\,}\,)$ with $[\ \,_\mu\ \, ]\, \tilde{\,} =(1\tt \de^*)\circ \Phi_\mu$. Now, we  have that $\varphi:(R, [\ \,_\mu\ \, ])  \to ((R^\c)^\c , [\ \,_\mu\ \, ]\, \tilde{\,}\,)$ is an isomorphism of finite free Lie conformal (super)algebras since
\begin{align*}
  ([\varphi(a)_\mu \varphi(b)]\, \tilde{\,}\,)_\la (f)
& =[\Phi_\mu( \varphi(a)\tt \varphi(b))]_\la (\de(f))
  =\textstyle{\sum} \, \big([\varphi(a)]_{\mu}(f_{(1)})\big) \
    \big([\varphi(b)]_{\la-\mu}(f_{(2)})\big)\\
& =\textstyle{\sum} \,  (f_{(1)})_{-\mu}(a) \
    (f_{(2)})_{-\la+\mu}(b) = [\Phi_{-\mu}(\de(f))]_{-\la} (a\tt b)
  = [\pi_{\mu}(f)]_{-\la} (a\tt b)\\
& = f_{-\la}([a_\mu b])=(\varphi([a_\mu b]))_\la (f),
\end{align*}
for any $a,b\in R$, and $f\in R^\c$.  It is clear that if $h:R_1\to R_2$ is a homomorphism of   Lie conformal (super)algebras, then we have
$ \varphi\circ h= F(G(h))\circ \varphi$. Hence, we obtained that $F\circ G\cong 1_\mathcal{L}$, where $\mathcal{L}$ is the category of finite free Lie conformal (super)algebras, finishing the proof.
\end{proof}

\vskip .2cm

\begin{proposition}\label{ideal}
  Using the previous construction, the ideals of a finite free Lie conformal (super)algebra correspond to sub(super)coalgebras of the associated finite free  differential Lie (super)coalgebra. Therefore, simple finite Lie conformal (super)algebras correspond to simple finite differential Lie (super)coalgebras.
\end{proposition}

\begin{proof}
   Let \( R \) be a finite free Lie conformal superalgebra and \( C = G(R) \). The isomorphism \( \varphi : R \to C^{\c} \) allows us to identify \( R \) with the conformal dual of \( C \). For any \( \mathbb{C}[\partial] \)-submodule \( I \subseteq R \), define its annihilator:
\[
I^\perp = \{ f \in C : f_\lambda(i) = 0 \text{ for all } i \in I \}.
\]
Similarly, for a \( \mathbb{C}[\partial] \)-submodule \( J \subseteq C \), define:
\[
J^\perp = \{ a \in R : f_\lambda(a) = 0 \text{ for all } f \in J \}.
\]
These are \( \mathbb{C}[\partial] \)-submodules, and under the identification \( R \cong C^{\c} \),   the map \( I \mapsto I^\perp\) is a bijection between \( \mathbb{C}[\partial] \)-submodules of \( R \) and \( \mathbb{C}[\partial] \)-submodules of \( C \). We will show that this bijection restricts to one between ideals and subcoalgebras.

Now, let $I$ be an ideal of $R$. Then the quotient map $\pi: R \to R/I$ is a surjective homomorphism of Lie conformal algebras. Its dual $\pi^*: (R/I)^\c \to R^\c$ is an injective homomorphism of differential Lie coalgebras. Now, we prove that the image of $\pi^*$ is exactly $I^\perp$.  
For any $f \in (R/I)^\c$, we have $\pi^*(f) = f \circ \pi$. If $x \in I$, then $\pi(x) = 0$, so $f_\la (\pi(x)) = 0$, hence $\pi^*(f) \in I^\perp$, and conversely, if $g \in I^\perp$, define $\tilde{g}: R/I \to \mathbb{C}[\la]$ by $\tilde{g}_\la(a + I) = g_\la(a)$. This is well-defined since $g$ vanishes on $I$. Then $\pi^*(\tilde{g}) = g$, so $g \in \operatorname{Im}(\pi^*)$.

Since $(R/I)^\c$ is a differential Lie coalgebra and $\pi^*$ is a coalgebra morphism, its image $I^\perp$ is a subcoalgebra of $R^\c$.

Conversely, let $J$ be a subcoalgebra of $R^\c$. The inclusion $i: J \hookrightarrow R^\c$ is a differential Lie coalgebra homomorphism. Its dual $i^*: (R^{\c})^\c \to J^\c$ is a Lie conformal  algebra homomorphism. Using the natural isomorphism $\varphi:R \to (R^{\c})^\c$, we get a Lie conformal  algebra homomorphism ${i}^* \circ\varphi : R \to J^\c$. The kernel of ${i}^* \circ\varphi $ is an ideal of $R$. Using that $[({i}^* \circ\varphi )(a)]_\la(f) = [\varphi (a)]_\la(i(f)) = f_{-\la}(a)$, for all $f\in J, a\in R$, one checks that $\ker({i}^* \circ\varphi ) = J^\perp$. Moreover, these constructions are inverse to each other: if we start with an ideal $I$, then $I^\perp$ is a subcoalgebra, and $(I^\perp)^\perp = I$; if we start with a subcoalgebra $J$, then $J^\perp$ is an ideal, and $(J^\perp)^\perp = J$, finishing the proof of the first part. Since for any finite free Lie conformal superalgebra \( R \), we have $[R_\la R]\neq 0$ if and only if $\de_{R^\c}\neq 0$, simplicity is preserved, and the proposition is proved.
\end{proof}

\section{classification of Lie conformal superalgebras and differential Lie supercoalgebras}\lbb{classification}

The main (and highly non-trivial) result in \cite{DK} is that any finite simple Lie conformal algebra is
isomorphic either to Cur$(\fg)$, where $\fg$ is a simple
finite-dimensional Lie algebra, or to the Virasoro conformal
algebra, where
the {\it Virasoro conformal algebra} is defined by:
\begin{displaymath}
{\rm Vir}=\cp \,L, \qquad \quad [L_\la L]=(\p +2\la)\, L,
\end{displaymath}
and for  $\fg$  a  Lie (super)algebra, the {\it current conformal (super)algebra}
associated to $\fg$ is defined by:
\begin{displaymath}
{\rm Cur }(\fg)=\cp \otimes \fg, \qquad \quad [a_\la b]=[a,b],
\qquad a,b\in \fg.
\end{displaymath}
\vskip .2cm

\noindent Using the correspondence of categories, we obtain the following classification theorem.
\begin{theorem}
  Any finite simple differential Lie coalgebra is isomorphic to one of the following list:

1. $(\mathrm{Vir})^{*_c}=\cp L^*$, with $\delta(L^*)=\d L^*\tt L^*-L^*\tt \d L^*$.

2. $(\mathrm{Cur}(\mathfrak{g}))^{*_c}=\cp\tt \mathfrak{g}^*$, where $(\mathfrak{g^*},\bar{\delta})$ is a simple finite-dimensional Lie coalgebra,

\ \ \ and
\begin{equation*}
\delta(f):=\bar{\delta}(f)\qquad \hbox{ for all }f\in \mathfrak{g}^*.
\end{equation*}
\end{theorem}

On the other hand, the list of finite simple Lie conformal superalgebras is
much richer, mainly due to existence of several series of super
extensions of the Virasoro conformal algebra. The classification was obtained in \cite{FK}.

The first series is  $W_n$, with
$(n\geq 1)$. More precisely, let $\Lambda(n)$ be the Grassmann
superalgebra in the $n$ odd indeterminate $\xi_1, \xi_2,\ldots ,
\xi_n$, and $W(n)$ the Lie superalgebra of all derivations of $\Lambda(n)$. The Lie conformal superalgebra $W_n$ is
defined as
\begin{equation} \label{eq:3.2}
W_n=\CC[\p]\otimes \left(W(n)\oplus\Lambda(n)\right).
\end{equation}
The $\la$-bracket is defined as follows $(a,b\in W(n); f,g\in
\Lambda(n))$:
\begin{equation} \label{eq:3.3}
[a_\la b]=[a,b], \quad [a_\la f]= a(f)-(-1)^{p(a)p(f)}\la fa,\quad
[f_\la g]=-(\p +2\la )fg
\end{equation}
The Lie conformal algebra $W_n$ is simple for $n\geq 0$ and has
rank $(n+1)2^n$.

The second series is $S_n$. In order to describe it, we
need  the notion of conformal divergence. It is a $\CC[\p]$-module map $div: W_n \to \
$Cur$\ \Lambda (n)$, given by
\begin{displaymath}
div\ a =\sum_{i=1} (-1)^{p(f_i)} \p_i f_i, \qquad \quad div \ f
=-\p \, f,
\end{displaymath}
where $a=\sum_{i=1}^n f_i \p_i\in  W(n)$ and $f\in \Lambda (n)$.
The following identity holds in $\CC[\p]\otimes \Lambda (n)$,
where $D_1, D_2\in W_n$:
\begin{equation} \label{eq:div}
div \ [{D_1}_\la D_2] = (D_1)_\la (div\ D_2) - (-1)^{p(D_1)
p(D_2)} (D_2)_{-\la-\p} (div \ D_1).
\end{equation}
Therefore,
\begin{equation}\label{SSSSS}
  S_n=\{ D\in W_n\ : \ div \ D=0\}
\end{equation}

\vskip .1cm

\noindent is a subalgebra of the Lie conformal superalgebra $W_n$. It is
known that $S_n$ is simple for $n\geq 2$, and finite of rank $n
2^n$.

The next series is a family that include $S_n$. Let $D=\sum_{i=1}^n P_i(\p , \xi) \p_i + f(\p, \xi)$ be
an element of $W_n$. We define the deformed divergence as
\begin{displaymath}
div_bD=div D + bf.
\end{displaymath}
It still satisfies equation \ref{eq:div}, therefore
\begin{displaymath}
S_{n,b}= \{D\in W_n \, |\,  div_b \, D=0\}
\end{displaymath}
is a subalgebra of $W_n$, which is simple for $n\geq 2$ and has
rank $n2^n$. Observe that $S_{n,0} = S_n$.

The
 Lie conformal superalgebra $\tilde S_n$ is constructed
explicitly as follows:
\begin{displaymath}
\tilde S_n=\{D\in W_n \ |\  div((1+\xi_1\dots
\xi_n)D)=0\}=(1-\xi_1\dots \xi_n)S_n.
\end{displaymath}
The Lie conformal superalgebra $\tilde S_n$ is simple for $n\geq
2$ and has rank $n2^n$.

The  Lie conformal superalgebra $K_n$ is
identified with
\begin{equation} \label{eq:k1}
K_n=\CC[\p]\otimes \Lambda(n),
\end{equation}
and the $\la$-bracket for
$f=\xi_{i_1} \dots \xi_{i_r}, g=\xi_{j_1} \dots \xi_{j_s}$
being as follows \cite{FK}:
\begin{equation} \label{eq:k2}
[f_\la g]=\bigg( (r-2)\p (fg) + (-1)^r \sum_{i=1}^n (\p_i f)(\p_i
g)\bigg) + \lambda (r+s-4)fg.
\end{equation}
The Lie conformal superalgebra $K_n$ has rank $2^n$ over $\CC[\p]$.
It is simple for $n\geq 0, n\neq 4$, and the derived algebra $K_4'$ is simple
and has codimension 1 in $K_4$. More precisely,
\begin{equation*}
  K_4 = K_4' \oplus \CC \xi_1\dots\xi_4.
\end{equation*}

The Lie conformal superalgebra $CK_6$ is defined as the subalgebra
of $K_6$ given by (cf. \cite{CK2}, Theorem 3.1)
\begin{equation*}
    CK_6=\cp \hbox{-span}\  \{f+\beta(-1)^{\frac{|f|(|f|+1)}{2}}(-\p)^{3-|f|}
    \, f^{\, \bullet}\, : \, f\in \La (6), \ 0\leq |f|\leq 3\}.
\end{equation*}
\vskip .1cm

\noindent where $\beta^2=-1$ and for  a monomial $f=\xi_{i_1} \dots \xi_{i_r}\in \La (n)$, we let ${f}^{\, \bullet}$ be its
Hodge dual, i.e. the unique monomial in $\La (n)$  such that
${f} f^{\, \bullet}= \xi_1 \dots \xi_n$.
This definition of Hodge dual corresponds to the one in \cite
{CK2} or \cite{CL}, pp. 922. We shall use Hodge dual only for $CK_6$.

\vskip .3cm

The following theorem is the main result of \cite{FK}.

\begin{theorem}
Any finite simple Lie conformal superalgebra is isomorphic to one of the Lie conformal superalgebras of the following list:

1. $W_n \ \ (n\geq 0)$;

2. $S_{n,a}\ \ (n\geq 2, \ a \in \CC)$;

3. $\tilde S_n \ \ (n \ \mathrm{even}, \ n \geq 2)$;

4. $K_n \  \ (n \geq 0, n\neq 4)$;

5. $K'_4$;

6. $CK_6$;

7. $Cur(\mathfrak{g})$, where $\mathfrak{g}$ is a simple finite-dimensional Lie superalgebra.
\end{theorem}

Using the  correspondence of categories in Theorem \ref{equival}, Proposition \ref{ideal}
and the construction given in Theorem \ref{prop:dual}, we have one of the main result of this work:

\begin{theorem}
Any  simple differential Lie supercoalgebra of finite rank is isomorphic to one  of the following list:

1. $(W_n)^{*_c} \ \ (n\geq 0)$;

2. $(S_{n,a})^{*_c}\ \ (n\geq 2, \ a \in \CC)$;

3. $(\tilde S_n)^{*_c} \ \ (n \ \mathrm{even}, \ n \geq 2)$;

4. $(K_n)^{*_c} \  \ (n \geq 0, n\neq 4)$;

5. $(K'_4)^{*_c}$;

6. $(CK_6)^{*_c}$;

7. $(Cur(\mathfrak{g}))^{*_c}$, where $\mathfrak{g}$ is a simple finite-dimensional Lie superalgebra.
\end{theorem}

In the following section we present an explicit description of each of  them.

\section{Simple differential Lie   supercoalgebras of finite rank}\lbb{Lie co}

We shall frequently use the following notations:
\begin{equation} \label{eq:xi}
\xi_I:=\xi_{i_1}\cdots \xi_{i_s}\in \Lambda(n), \quad \hbox{ with }
I=\{i_1,\ldots , i_s\},
\end{equation}
and for any $J\subseteq \{1,\dots,n\}$ and $i\in J$, we define
\begin{equation}\label{epsi}
  \varepsilon_i^J:=\#\{j\in J\, :\, j<i\}.
\end{equation}
This is used to write $\d_i(\xi_J)=(-1)^{\varepsilon_i^J}\xi_{J-\{i\}}$.
We shall consider a basis of $\Lambda(n)$ given by the elements $\xi_I=\xi_{i_1}\cdots \xi_{i_s}$ with $i_1<\ldots < i_s$, therefore we need to write  $\xi_I \xi_J$ as an element of this basis, that is:
\begin{equation}\label{WW1}
  \xi_I\xi_J=
  \left\{
    \begin{array}{ll}
       (-1)^{\alpha(I,J)}\, \xi_{\mathrm{ord}(I,J)}, & \hbox{\ \  if }I\cap J =\emptyset; \\
     \qquad 0 , & \hbox{ \ if }I\cap J \neq \emptyset;
    \end{array}
  \right.
\end{equation}
where ord$(I,J)$ is the ordered set associated to $I\cup J$, and $\alpha(I,J)$ is the number of single permutations needed to order the juxtaposition of sets $I\cdot J$, when $I\cap J=\emptyset$. For example $\xi_{\{3,6\}}\xi_{\{1,2,4\}}=-\xi_{\{1,2,3,4,6\}}$. We shall use that
\begin{equation}\label{alfa}
  (-1)^{\al(I,J)}=(-1)^{\al(\!J\, ,\,I) + |I| |J|}.
\end{equation}

\vskip .2cm

\subsection{Differential Lie   supercoalgebra $(W_n)^{*_c}$}\lbb{Wn}

Recall from (\ref{eq:3.2}) that
\begin{equation*}
W_n=\CC[\p]\otimes \left(W(n)\oplus \Lambda(n)\right).
\end{equation*}
We denote by $\xi_I^*$ and $(\xi_I\d_i)^*$ the elements in the dual basis of $\Lambda(n)$ and $W(n)$, respectively.  In order to apply (\ref{eq:De}), we need to compute the $\la\,$-brackets (\ref{eq:3.3}) in $W_n$ in terms of the basis in $\Lambda(n)$ and $W(n)$. We have:
\begin{equation*}
  [\,{\xi_I}_\la \,\xi_J\,]=
   \left\{
     \begin{array}{ll}
       -(2\la+\d)\, (-1)^{\alpha(I,J)}\, \xi_{\mathrm{ord}(I,J)}, & \hbox{ \ \ if } I\cap J =\emptyset;\\
       \qquad \qquad 0, & \hbox{\ \ \ if } I\cap J \neq\emptyset;
     \end{array}
   \right.
\end{equation*}
and
\begin{equation*}
  [\,{\xi_I\d_i}_\la \,\xi_J\,]= \xi_I \d_i(\xi_J)-\la (-1)^{(|I|+1)|J|}\xi_J\xi_I\d_i,
\end{equation*}
where
\begin{equation*}
  \xi_I \d_i(\xi_J)=\left\{
                      \begin{array}{ll}
                        (-1)^{\varepsilon_i^J+\, \al(I,J-\{i\})} \,  \xi_{\mathrm{ord}(I,J-\{i\})}, & \hbox{\ \  if } i\in J \hbox{ and } I\cap (J-\{i\})=\emptyset.\\
                       \qquad \qquad 0 , & \hbox{\ \  otherwise.}
                      \end{array}
                    \right.
\end{equation*}
and the expression of $\xi_J\xi_I\d_i$ follows by (\ref{WW1}). Similarly, one can describe
\vskip -.1cm
\begin{equation*}
  [\,{\xi_I\d_i}_\la \,\xi_J\d_j\,]= \xi_I \d_i(\xi_J)\d_j- (-1)^{(|I|+1)(|J|+1)}\xi_J\d_j(\xi_I)\d_i,
\end{equation*}
\vskip .1cm
\noindent in terms of the basis in $W(n)$. Now, using (\ref{eq:De}), we can write the co-product in $(W_n)^\c$.

\

\begin{theorem}
  The Lie supercoalgebra structure on $(W_n)^\c$ is defined by the following formulas, in the dual basis:

  \begin{align}\label{W1}
    \de (\xi_K^*) &= \sum_{\substack{I,J \, : \, I\cap J=\emptyset \\
    \mathrm{ord}(I,J)=K}}  (-1)^{\alpha(I,J)}\, \Big[\, \xi_I^*\tt \d \, \xi_J^*-
    (-1)^{|I|\,|J|} \,
    \d \,\xi_J^*\tt \xi_I^*\Big] \\
     &
     + \sum_{\substack{I,J,\, i \, : \, i\in J  \\ I\cap (J-\{i\})=\emptyset \\
    \mathrm{ord}(I,J-\{i\})=K}}  (-1)^{\varepsilon_i^J+\, \alpha(I,J-\{i\})}\, \Big[(\xi_I \d_i)^*\tt  \xi_J^*- (-1)^{|J|(|I|+1)}  \xi_J^*\tt (\xi_I \d_i)^*
    \Big]\nonumber
  \end{align}
and
\begin{align}\label{W2}
    \de &((\xi_K \d_k)^*) = \sum_{\substack{I,J \, : \, I\cap J=\emptyset \\
    \mathrm{ord}(I,J)=K}}  (-1)^{\alpha(I,J)}\, \Big[\, \xi_I^*\tt \d \, (\xi_J\d_k)^*-
    (-1)^{|I|(|J|+1)} \,
    \d \,(\xi_J\d_k)^*\tt \xi_I^*\Big] \\
    + &
      \sum_{\substack{I,J,\, i \, : \, i\in J  \\ I\cap (J-\{i\})=\emptyset \\
    \mathrm{ord}(I,J-\{i\})=K}}  (-1)^{\varepsilon_i^J+\, \alpha(I,J-\{i\})}\, \Big[(\xi_I \d_i)^*\tt  (\xi_J\d_k)^*- (-1)^{(|I|+1)(|J|+1)}  (\xi_J\d_k)^*\tt (\xi_I \d_i)^*
    \Big].\nonumber
  \end{align}
\end{theorem}

\vskip .4cm

\begin{proof}
\noindent Using (\ref{eq:De}), for any $f\in (W_n)^\c$, we have
\begin{eqnarray*}
  & &\de(f) = \sum_{I,J}\, \Big( f_\mu([\,{\xi_I}_\la \, \xi_J])\Big) (\xi_I^*\tt\xi_J^*)
\,  +   \sum_{\substack{j=1\dots n \\ I,J}}\, \Big( f_\mu([\,{\xi_I}_\la \, \xi_J\d_j])\Big) (\xi_I^*\tt(\xi_J\d_j)^*) \\
   &+&  \sum_{\substack{i=1\dots n \\ I,J}}\, \Big( f_\mu([\,{\xi_I\d_i}_\la \, \xi_J])\Big) ((\xi_I\d_i)^*\tt \xi_J^*)
\, + \sum_{\substack{i,j=1\dots n \\ I,J}}\, \Big( f_\mu([\,{\xi_I\d_i}_\la \, \xi_J\d_j])\Big) ((\xi_I\d_i)^*\tt(\xi_J\d_j)^*),
\end{eqnarray*}
where we must remember that we have to replace $
\la=\p\otimes 1$ and $\mu= -\p\otimes 1 - 1\otimes \p $.
Now, we use the description of the $\la$-brackets:
\begin{eqnarray*}
  \de(f) & = & \sum_{I,J }\, \Big( f_\mu \left( (-1)^{\alpha(I,J)+1}\, (2\la+\d)\, \xi_{\mathrm{ord}(I,J)}\right)\Big) (\xi_I^*\tt\xi_J^*) \\
\,
    & - & \sum_{\substack{j=1\dots n \\ I,J}}\, (-1)^{|I|(|J|+1)}\Big( f_\mu\left(\xi_J \d_j({\xi_I})+(-1)^{|I|(|J|+1)}(\la+\d) \, \xi_I\xi_J\d_j\right)\Big) (\xi_I^*\tt(\xi_J\d_j)^*) \\
   &+&  \sum_{\substack{i=1\dots n \\ I,J}}\,
\Big( f_\mu\left(\xi_I \d_i({\xi_J})-(-1)^{(|I|+1)|J|}\la \, \xi_J\xi_I\d_i\right)\Big) ((\xi_I\d_i)^*\tt \xi_J^*)  \\
   &+& \sum_{\substack{i,j=1\dots n \\ I,J}}\,
\Big( f_\mu\left(\xi_I \d_i({\xi_J})\d_j-(-1)^{(|I|+1)(|J|+1)} \, \xi_J\d_j(\xi_I)\d_i\right)\Big) ((\xi_I\d_i)^*\tt(\xi_J\d_j)^*).
\end{eqnarray*}
In particular, if we take $f=\xi_K^*$ (in the dual basis) and we replace $
\la=\p\otimes 1$ and $ \mu= -\p\otimes 1 - 1\otimes \p $, we obtain:
\begin{eqnarray*}
  \de(\xi_K^*) & = & \sum_{\substack{I,J\, :\, I\cap J=\emptyset\\ \mathrm{ord}(I,J)=K}}\,  (-1)^{\alpha(I,J)}\, \left(1\tt\d -\d\tt 1 \right) (\xi_I^*\tt\xi_J^*) \\
\,
    & - & \sum_{\substack{I,J,\, j \, :\, j\in I \\ J\cap (I-\{j\})=\emptyset\\
\mathrm{ord}(J,I-\{j\})=K}}
\, (-1)^{|I|(|J|+1)+\varepsilon_j^I+\al(J,I-\{j\})} (\xi_I^*\tt(\xi_J\d_j)^*) \\
   &+&  \sum_{\substack{ I,J,\, i \, :\, i\in J \\I\cap (J-\{i\})=\emptyset\\
\mathrm{ord}(I,J-\{i\})=K}}
\, (-1)^{\varepsilon_i^J+\al(I,J-\{i\})}  ((\xi_I\d_i)^*\tt \xi_J^*).
\end{eqnarray*}
Observe that if we use (\ref{alfa}) in the first sum, we get the first sum in (\ref{W1}), and if we interchange $i,I$ with $j,J$ in the second sum, we can put the last two sums together, obtaining the last term in (\ref{W1}), finishing the first part of the theorem.

Now, we take $f=(\xi_K\d_k)^*$, and replace ${
\la=\p\otimes 1,\ \mu= -\p\otimes 1 - 1\otimes \p }$, we obtain:
\begin{eqnarray*}
  \de((\xi_K\d_k)^*) & = & \sum_{\substack{I,J\, :\, I\cap J=\emptyset\\ \mathrm{ord}(I,J)=K}}\,  (-1)^{\alpha(I,J)}\, (1\tt\d)\, (\xi_I^*\tt(\xi_J\d_k)^*) \\
\,
    & - & \sum_{\substack{I,J\, :\, I\cap J=\emptyset\\  \mathrm{ord}(J,I)=K}}\,  (-1)^{|J|(|I|+1)+\alpha(J,I)}\, (\d\tt 1)\, ((\xi_I\d_k)^*\tt \xi_J^*) \\
   &+&  \sum_{\substack{I,J,\, i \, :\, i\in J \\ I\cap (J-\{i\})=\emptyset\\
 \mathrm{ord}(I,J-\{i\})=K}}
\, (-1)^{\varepsilon_i^J+\al(I,J-\{i\})}  ((\xi_I\d_i)^*\tt (\xi_J\d_k)^*) \\
    & - &
\sum_{\substack{I,J,\, j \, :\, j\in I \\ J\cap (I-\{j\})=\emptyset\\
 \mathrm{ord}(J,I-\{j\})=K}}
\, (-1)^{(|I|+1)(|J|+1)+\varepsilon_j^I+\al(J,I-\{j\})} ((\xi_I\d_k)^*\tt(\xi_J\d_j)^*) \\
\end{eqnarray*}
Observe that  if  we interchange $i,I$ with $j,J$ in the second sum, then the first two sums produce  the first sum in (\ref{W2}), and we can put the last two sums together by interchanging $i,I$ with $j,J$, obtaining the last term in (\ref{W2}), finishing the proof.
\end{proof}

\subsection{Differential Lie   supercoalgebra $(S_n)^{*_c}$}\lbb{Sn}

Recall from (\ref{SSSSS}) that
\begin{equation*}
 S_n=\{ D\in W_n\ : \ div \ D=0\}.
\end{equation*}
For any $\xi_I\in \Lambda(n)$, with $|I|<n$, we define
\begin{equation*}
  B_I=( |I|-n )\, \xi_I + \p\, \sum_{\substack{i=1 \\ i\notin I}}^{n}\ (\xi_I \xi_i) \p_i.
\end{equation*}
A simple computation shows that $B_I\in S_n$. For any $I, i\notin I$ and $j\notin I$, we also consider the following elements in $S_n$
\begin{align*}
  A_{I,i} &  =\xi_I\, \p_i  \\
  A_{I ,i, j} &  =\xi_I \,(\xi_i\, \p_i -  \xi_j\, \p_j).
\end{align*}
We shall frequently use the following notations: for any $I\subseteq \{1,\dots,n\}$, we denote by $I^c$ the complement of $I$ in $\{1,\dots,n\}$.

Observe that
 \begin{equation}\label{bbaassee}
   \big\{A_{I,i}\big\}\cup \big\{ A_{I ,  i_l, i_{l+1}}\, | \  l=1,\dots , k-1, \ \mathrm{ with }\ \,  I^c=\{i_1,\dots , i_k\}, \mathrm{ and }\  i_1<\cdots < i_k \big\},
 \end{equation}
is a basis (over $\CC$) of $S(n)$, the $0$-divergence subalgebra of $W (n)$, with $(n-1)2^n+1$ elements.

Then, by \cite{FK}, p.38, we have that the elements in (\ref{bbaassee}), together with $\{B_I\}$, form a basis of $S_n$ as a $\CC[\p]$-module.

In order to apply (\ref{eq:De}), we need to compute the $\la\,$-brackets (\ref{eq:3.3}) in $S_n$ in terms of these generators.

\begin{proposition}
  The $\la\,$-brackets  in $S_n$ are the following.

\vskip .2cm

\noindent $\bullet$ If $i,j\notin I$ and $l,m\notin J$:
\begin{equation*}
  \big[ A_{I,i,j}\,\,_\la \, A_{J,l,m} \big] =  0.
\end{equation*}
\noindent  $\bullet$ If $i,j\notin I$ with $i<j$, and $r\notin J$:
\begin{equation*}
  \big[ A_{I,i,j}\,\,_\la \, A_{J,r}\big]
=\left \{
    \begin{aligned}
     \hbox{\small $ { (-1)^{\varepsilon_r^I +1+|I| + |J|}  \,
\xi_{I-\{r\}} \, \xi_J \, (\xi_i\p_i-\xi_j\p_j)}$}
&,\quad \ \hbox{\footnotesize $ { \text{if} \ \ \ r\in I; i,j\notin
  (I-\{r\})\cup J;
I\cap J=\emptyset}$} \\
\tilde\de_{r,j}\ \xi_I \xi_J \p_j - \tilde\de_{r,i}\  \xi_I \xi_J \p_i \  ,\qquad\ \
&\quad \ \  \hbox{\footnotesize $ {  \text{if} \   \  \ r\notin I; I\cap J=\emptyset}$} \\
      0 \, ,\qquad\qquad\qquad\qquad  &\quad \  \  \hbox{\footnotesize $ {  \text{otherwise} }$}
    \end{aligned}
  \right .
\end{equation*}
\noindent where $\tilde\delta_{i,j}$ is the Kronecker delta.

\vskip .2cm

\noindent  $\bullet$ If $i\notin I$, and $j\notin J$:
\begin{equation*}
  \big[ A_{I,i}\,\,_\la \, A_{J,j}\big]
=\left \{
    \begin{aligned}
 0\, ,\qquad\qquad\qquad \qquad
&\quad \    \text{if}  \ \ \ i\notin J; j\notin I\qquad  \qquad\\
     (-1)^{\varepsilon_i^J }  \,
\xi_{I} \, \xi_{J-\{i\}} \, \p_j \, ,\  \ \ \qquad \qquad
&\quad \   \text{if} \ \ \   i\in J; j\notin I \qquad \qquad \\
(-1)^{\varepsilon_j^I + |I|}  \,
\xi_{I-\{j\}} \, \xi_{J} \,  \p_i  \, , \ \ \qquad \qquad
&\quad \   \text{if} \ \ \   i\notin J; j\in I  \\
       (-1)^{\varepsilon_i^J +\varepsilon_j^I+|I| + |J|}  \,
\xi_{I-\{j\}} \, \xi_{J-\{i\}} \, (\xi_j\p_j-\xi_i\p_i) \, , \ \
& \quad \   \text{if} \ \ \ i\in J; j\in
  I
    \end{aligned}
  \right .
\end{equation*}

\vskip .2cm

\noindent  $\bullet$  If $I\cap J\neq \emptyset$, then $\big[ B_{I}\,\,_\la \, B_{J}\big]=0$.

\hskip -.05cm If $I\cap J= \emptyset$, then
\begin{equation*}
  \big[ B_{I}\,\,_\la \, B_{J}\big]=\la \, (2n-|I|-|J|)\, B_{I\cdot J} \, +\,
\p (n-|I|)\, B_{I\cdot J}.
\end{equation*}

\vskip .2cm

\noindent  $\bullet$ 
 If $I\cap J\neq \emptyset$, then $\big[ A_{I,i}\,\,_\la \, B_{J}\big]=0$.

\hskip -.05cm If $I\cap J= \emptyset$, then:

\noindent if $i\notin I$ and $i\notin J$:
\begin{equation*}
  \big[ A_{I,i}\,\,_\la \, B_{J}\big]= (-1)^{|J|} \big[ \la\, (n-|J|+1-|I|) + \p\, (1-|I|)\big]\, \xi_I\xi_J \p_i,
\end{equation*}
and if $i\notin I$ and $i\in J$:
\begin{align*}
  \big[ A_{I,i}\,\,_\la \, B_{J}\big] = 
  &
\left(\textstyle{\frac{|J|-n }{|I|+|J|-n-1}}\right)  \,  (-1)^{\varepsilon_i^J+\al(I,J-\{i\})} \, B_{\mathrm{ord}(I,J-\{i\})}\\
& 
+ \p   \left(\textstyle{\frac{|I|-1 }{|I|+|J|-n-1}}\right)    (-1)^{\varepsilon_i^J+\al(I,J-\{i\})}  \,
 \sum_{j\notin  I\cup (J-\{i\})}  \ A_{\mathrm{ord}(I,J-\{i\}),j,i} \\
& 
+ \la \, (-1)^{\varepsilon_i^J+\al(I,J-\{i\})} \sum_{\substack{j\notin I\cup J  }}   \ A_{\mathrm{ord}(I,J-\{i\}),j,i} .
\end{align*}

\vskip .2cm

\noindent  $\bullet$ If $i,k \notin I$:
\begin{equation*}
  \big[ A_{I,i,k}\,\,_\la \, B_{J}\big]
=\left \{
    \begin{aligned}
 0\qquad\qquad\qquad \qquad
&,\,\,\,  \text{if}  \ \  i,k\in J\qquad  \qquad\\
     (\la (n-|J|-|I|)-\p |I|)\,\xi_I\xi_J(\xi_i\p_i-\xi_k\p_k)\  \ \
&,\,\,\,  \text{if} \ \    i\notin J; k\notin J \qquad \qquad \\
\left(\textstyle{\frac{ |J|-n}{|I|+|J|-n}}\right) B_{IJ}+
\left(\la +\p\, \textstyle{\frac{ |I|}{|I|+|J|-n}}\right)
\dsum_{j\notin I\cup J}\xi_I\xi_J(\xi_j\p_j-\xi_k\p_k)
&,\,\,\, \text{if} \ \    i\in J; k\notin J  \\
-\left(\textstyle{\frac{ |J|-n}{|I|+|J|-n}}\right) B_{IJ}-
\left(\la +\p\, \textstyle{\frac{ |I|}{|I|+|J|-n}}\right)
\dsum_{j\notin I\cup J}\xi_I\xi_J(\xi_j\p_j-\xi_i\p_i)
&,\,\,\,  \text{if} \ \    i\notin J; k\in J
    \end{aligned}
  \right .
\end{equation*}
where the previous terms are 0 if $I\cap J\neq \emptyset$.
\end{proposition}

\

\noindent The proof of the previous Proposition follows by straightforward computations.
 Now, using (\ref{eq:De}), we can write the co-product in $(S_n)^\c$.

\

\begin{theorem}
  The Lie supercoalgebra structure on $(S_n)^\c$ is defined by the following formulas, in the dual basis:
  
  \vskip .2cm 
  
\noindent $\bullet$  For all $K$ with $|K|<n$,

 \begin{align}\label{S1}
 & \de(B_K^*) =
 \sum_{\substack{I,J\, :\, I\cap J=\emptyset\\  \mathrm{ord}(I,J)=K}}\,  (-1)^{\alpha(I,J)}\,  (n-|J|)\, \Big[ \p (B_I^*)\tt B_J^* - (-1)^{|I||J|}\, B_J^*\tt \p (B_I^*)\Big]
\\
   & +
\sum_{\substack{ I,J,i \, : \,i\in J, \, i\notin I  \\ I\cap (J-\{i\})=\emptyset \\
     \mathrm{ord}(I,J-\{i\})=K}}
\left(\textstyle{\frac{ |J|-n}{|I|+|J|-n-1}}\right)
 (-1)^{\varepsilon_i^J+\, \alpha(I,J-\{i\})} \Big[A_{I,i}^*\tt  B_J^*- (-1)^{(|I|+1)|J|}  B_J^*\tt A_{I,i}^*
    \Big]\nonumber
\\
   & +
\sum_{\substack{ I,J \, : \, 
    \mathrm{ord}(I,J)=K}} \ \ 
\sum_{\substack{ r \, : \, i_r\in J, \, i_{r+1}\notin J}}    
\left(\textstyle{\frac{ |J|-n}{|I|+|J|-n}}\right)
(-1)^{\alpha(I,J)} \Big[A_{I,i_r,i_{r+1}}^*\tt  B_J^*   - (-1)^{|I||J|}  B_J^*\tt A_{I,i_r,i_{r+1}}^*
    \Big]\nonumber
    \\
& - 
\sum_{\substack{ I,J \, : \, 
    \mathrm{ord}(I,J)=K}} \ \ 
\sum_{\substack{ r \, : \, i_r\notin J, \, i_{r+1}\in J}}    
\left(\textstyle{\frac{ |J|-n}{|I|+|J|-n}}\right)
(-1)^{\alpha(I,J)} \Big[A_{I,i_r,i_{r+1}}^*\tt  B_J^*   - (-1)^{|I||J|}  B_J^*\tt A_{I,i_r,i_{r+1}}^*
    \Big]\nonumber    
\end{align}
where in the last two double sums, we use the notation $I^c=\{i_1,\dots ,i_l\}$,  with $i_1<\dots < i_l$, where $I^c$ is the complement of $I$ in $\{1,\dots,n\}$.

  \vskip .2cm 
  
\noindent $\bullet$  For all $K$, and $k\notin K$,

\begin{align}\label{S2}
    \de(A_{K,k}^*) 
=&
\sum_{\substack{I,J \, : \, I\cap J=\emptyset \\
    \mathrm{ord}(I,J)=K\\ k\notin I\cup J, \,k=i_r}}  (-1)^{\alpha(I,J)}\,
\Bigg\{\,- A_{I,k,{i_{r+1}}}^*\tt  A_{J,k}^*
+
A_{I,{i_{r-1}},_k}^*\tt  A_{J,k}^*\\
& \qquad \qquad \qquad \hskip 2cm +
    (-1)^{|I|(|J|+1)} \,
   \Big[A_{J,k}^*\tt A_{I,k,{i_{r+1}}}^*
 -
   A_{J,k}^* \tt A_{I,{i_{r-1}},_k}^*
\Big]
\Bigg\}
\nonumber\\
    + &
      \sum_{\substack{ I,J ,i\, : \,i\in J, \,i\notin I, \, k\notin I\cup J \\ I\cap (J-\{i\})=\emptyset \\
    \mathrm{ord}(I,J-\{i\})=K}}  (-1)^{\varepsilon_i^J+\, \alpha(I,J-\{i\})}\, \Big[A_{I,i}^*\tt  A_{J,k}^*- (-1)^{(|I|+1)(|J|+1)}  A_{J,k}^*\tt A_{I,i}^*
    \Big]\nonumber\\
    + &
 \sum_{\substack{I,J\, :\, I\cap J=\emptyset\\ \mathrm{ord}(I,J)=K}}\,  (-1)^{\alpha(I,J)+|J|}\, \Bigg\{
(n-|J|)\, \Big[ \p (A_{I,k}^*)\tt B_J^*
-
    (-1)^{(|I|+1)|J|} \, B_J^*\tt \p (A_{I,k}^*)
 \Big]
 \nonumber\\
& \qquad \qquad \qquad \qquad \qquad +
   (|I|-1)\,\Big[  A_{I,k}^*\tt \p(B_J^*)
-
    (-1)^{(|I|+1)|J|} \, (\p(B_J^*)\tt A_{I,k}^* )\Big]
 \Bigg\}.\nonumber
  \end{align}
    
  \vskip .2cm 
  
\noindent $\bullet$  
For  all $K$, and $k=1,\dots,l-1$, where $K^c= \{i_1,\dots ,i_l\}$,  with $i_1<i_2<\dots  < i_l$,
\begin{align}\label{3S}
    \de &(A_{K,{i_k},{i_{k+1}}}^*) = 
    \nonumber \\
&\sum_{\substack{I,J,l \, : \, I\cap J=\emptyset \\
    \mathrm{ord}(I-\{l\},J)=K\\ l\in I; \,l\notin J; \, i_k,i_{k+1}\notin I\cup J}}  
    (-1)^{\varepsilon_l^I+1+|I|+|J|+\alpha(I-\{l\},J)}\,
\Big[ A_{I,{i_k},{i_{k+1}}}^*\tt  A_{J,l}^* - (-1)^{|I|(|J|+1)} A_{J,l}^*\tt A_{I,{i_k},{i_{k+1}}}^*\Big]
\nonumber\\
     + &
      \sum_{\substack{ I,J \, ,i,j\, : i\notin I, \,i\in J, \,
j\in I, \,j\notin J
\\ \, (I-\{j\})\cap (J-\{i\})=\emptyset \\
    \mathrm{ord}(I-\{j\},J-\{i \})=K}}  
(-1)^{\varepsilon_{i}^J+\,\varepsilon_{j}^I+ |I|+|J|\, \alpha(I-\{j\},J-\{i\})}\, 
\Big(A_{K, i_k, {i_{k+1}}}^*\Big)_\mu \Big(A_{\mathrm{ord}(I-\{j\},J-\{i\}),j,i}\Big)\, 
\\
& \qquad \qquad \qquad \qquad \qquad \qquad
\times \Big[A_{I,i }^* \tt  A_{J,j}^*- (-1)^{(|I|+1)(|J|+1)}  A_{J,j}^*\tt A_{I,i }^*
    \Big]\nonumber\\    
 + &  \sum_{\substack{ I,J \, ,i\, :  \,
i \notin I, i\in J,   \, I\cap J =\emptyset \\ 
\mathrm{ord}(I,J-\{i\})=K
}} \,\,
\sum_{j\notin  I\cup J} 
\left(\textstyle{\frac{(-1)^{ \varepsilon_{i}^J+ \, \alpha(I,J-\{i\})}}{|I|+|J|-n-1}}\right)\,
(A^*_{K,i_k,i_{k+1}})_\mu 
 ( A_{\mathrm{ord}(I,J-\{i\}),j,i})
 \nonumber\\
    &\qquad \times
\Bigg\{ (1-|I|)\,
\Big[ A_{I,i}^*\tt \p (B_J^*)
        - (-1)^{(|I|+1)|J|}(\p (B_J^*)\tt A_{I,i}^*)\Big]\,
\nonumber\\
& \qquad \qquad \qquad \qquad
+ (|J|-n)\,
\Big[ \p(A_{I,i}^*)\tt B_J^*
        - (-1)^{(|I|+1)|J|}( B_J^* \tt\p(A_{I,i}^*))\Big]
    \Bigg\}\nonumber
\\
+ & \sum_{\substack{I,J\, :\, I\cap J=\emptyset\\ \mathrm{ord}(I,J)=K}}\,  (-1)^{\alpha(I,J)}\,
 \Bigg\{ \!
  (n-|J|)
\Big[ \p (A_{I,{i_k},{i_{k+1}}}^*)\tt B_J^*
 -
    (-1)^{|I||J|}
B_J^*\tt \p (A_{I,{i_k},{i_{k+1}}}^*)
\Big]\nonumber\\
&  \qquad \qquad  \qquad \qquad\qquad + |I|
\Big[ A_{I,{i_k},{i_{k+1}}}^*\tt \p(B_J^*)
 -
    (-1)^{|I||J|}
\, \p(B_J^*)\tt A_{I,{i_k},{i_{k+1}}}^*
\Big]
\Bigg\}\nonumber\\
+ & \sum_{\substack{I,J\, :\, I\cap J=\emptyset\\ \mathrm{ord}(I,J)=K}}   \
\sum_{\substack{r\,:\, j_r\in J,  j_{r+1}\notin J  }}  \left(\textstyle{\frac{(-1)^{\al(I,J)}}{|I|+|J|-n}}\right)\,
(A^*_{K,i_k,i_{k+1}})_\mu  \Big( \sum_{j\notin I\cup J}
A_{\mathrm{ord}(I,J),j,{j_{r+1}}}\Big) \nonumber\\
 & \qquad \times \Bigg\{ \!
  (|J|-n)
\Big[ \p (A_{I,{j_r},{j_{r+1}}}^*)\tt B_J^*
 -
    (-1)^{|I||J|}
B_J^*\tt \p (A_{I,{j_r},{j_{r+1}}}^*)
\Big]\nonumber\\
&  \qquad \qquad  \qquad \qquad\qquad - |I|
\Big[ A_{I,{j_r},{j_{r+1}}}^*\tt \p(B_J^*)
 -
    (-1)^{|I||J|}
\, \p(B_J^*)\tt A_{I,{j_r},{j_{r+1}}}^*
\Big]
\Bigg\}\nonumber\\
- & \sum_{\substack{I,J\, :\, I\cap J=\emptyset\\ \mathrm{ord}(I,J)=K}}   \
\sum_{\substack{r\,:\, j_r\notin J,  j_{r+1}\in J  }}  \left(\textstyle{\frac{(-1)^{\al(I,J)}}{|I|+|J|-n}}\right)\,
(A^*_{K,i_k,i_{k+1}})_\mu  \Big( \sum_{j\notin I\cup J}
A_{\mathrm{ord}(I,J),j,{j_{r}}}\Big)\nonumber \\
 & \qquad \times \Bigg\{ \!
  (|J|-n)
\Big[ \p (A_{I,{j_r},{j_{r+1}}}^*)\tt B_J^*
 -
    (-1)^{|I||J|}
B_J^*\tt \p (A_{I,{j_r},{j_{r+1}}}^*)
\Big]\nonumber\\
&  \qquad \qquad  \qquad \qquad\qquad - |I|
\Big[ A_{I,{j_r},{j_{r+1}}}^*\tt \p(B_J^*)
 -
    (-1)^{|I||J|}
\, \p(B_J^*)\tt A_{I,{j_r},{j_{r+1}}}^*
\Big]
\Bigg\},\nonumber
\end{align}
where in the last sums we use the notation $ I^c=\{j_1,\dots,j_s\}$, with $j_1<\cdots <j_s$, where $I^c$ is the complement of $I$ in $\{1,\dots,n\}$.
\end{theorem}

\vskip .5cm

\begin{proof}
For all $K$ with $|K|<n$, and using (\ref{eq:De}), we have
\begin{align}\label{sss1}
 & \de(B_K^*) =
\sum_{\substack{I,J\, :\, I\cap J=\emptyset\\ \mathrm{ord}(I,J)=K}}\,  (-1)^{\alpha(I,J)}\,  \Big[ \la (2n-|I|-|J|)+\mu (n-|I|)\Big] (B_I^*\tt B_J^*)
\\
   & \quad \,  +
\sum_{\substack{I,J,\, i \, : \, i\in J, \, i\notin I  \\ I\cap (J-\{i\})=\emptyset \\
    \mathrm{ord}(I,J-\{i\})=K}}
\left(\textstyle{\frac{ |J|-n}{|I|+|J|-n-1}}\right)
 (-1)^{\varepsilon_i^J+\, \alpha(I,J-\{i\})} \Big[A_{I,i}^*\tt  B_J^*- (-1)^{(|I|+1)|J|}  B_J^*\tt A_{I,i}^*
    \Big]\nonumber
\\
   &\quad \,   +
\sum_{\substack{ I,J \, : \, 
    \mathrm{ord}(I,J)=K}} \ \ 
\sum_{\substack{ r \, : \, i_r\in J, \, i_{r+1}\notin J}}    
\left(\textstyle{\frac{ |J|-n}{|I|+|J|-n}}\right)
(-1)^{\alpha(I,J)} \Big[A_{I,i_r,i_{r+1}}^*\tt  B_J^*   - (-1)^{|I||J|}  B_J^*\tt A_{I,i_r,i_{r+1}}^*
    \Big]\nonumber
    \\
& \quad \, - 
\sum_{\substack{ I,J \, : \, 
    \mathrm{ord}(I,J)=K}} \ \ 
\sum_{\substack{ r \, : \, i_r\notin J, \, i_{r+1}\in J}}    
\left(\textstyle{\frac{ |J|-n}{|I|+|J|-n}}\right)
(-1)^{\alpha(I,J)} \Big[A_{I,i_r,i_{r+1}}^*\tt  B_J^*   - (-1)^{|I||J|}  B_J^*\tt A_{I,i_r,i_{r+1}}^*
    \Big]\nonumber    
\end{align}
where we must remember that we have to replace $
\la=\p\otimes 1$ and $\mu= -\p\otimes 1 - 1\otimes \p $. Observe that the last three terms in (\ref{sss1}) are the last three terms in (\ref{S1}). If we replace ${
\la=\p\otimes 1,\ \mu= -\p\otimes 1 - 1\otimes \p }$ in the first sum, we obtain:
\begin{equation*}
  \sum_{\substack{I,J\, :\, I\cap J=\emptyset\\ \mathrm{ord}(I,J)=K}}\,  (-1)^{\alpha(I,J)}\,  \Big[  (n-|J|)\, (\p (B_I^*)\tt B_J^*)- (n-|I|)\, (B_I^*\tt \p (B_J^*))\Big],
\end{equation*}
and, using (\ref{alfa}), it becomes equal to
\begin{equation*}
  \sum_{\substack{I,J\, :\, I\cap J=\emptyset\\ \mathrm{ord}(I,J)=K}}\,  (-1)^{\alpha(I,J)}\,  (n-|J|)\, \Big[ \p (B_I^*)\tt B_J^* - (-1)^{|I||J|}\, B_J^*\tt \p (B_I^*)\Big],
\end{equation*}
proving (\ref{S1}). Similarly, we have for all $K$, and $k\notin K$
\begin{align}\label{sss2}
    \de &(A_{K,k}^*) = \\
=&
\sum_{\substack{I,J \, : \, I\cap J=\emptyset \\
    \mathrm{ord}(I,J)=K\\ k\notin I\cup J, \,k=i_r}}  (-1)^{\alpha(I,J)}
\Bigg\{\,- A_{I,k,{i_{r+1}}}^*\tt  A_{J,k}^*
+
A_{I,{i_{r-1}},_k}^*\tt  A_{J,k}^*\nonumber\\
& \qquad \qquad \qquad \hskip 2cm +
    (-1)^{|I|(|J|+1)} \,
   \Big[A_{J,k}^*\tt A_{I,k,{i_{r+1}}}^*
 -
   A_{J,k}^* \tt A_{I,{i_{r-1}},_k}^*
\Big]
\Bigg\} \nonumber\\
    + &
      \sum_{\substack{I,J,\, i \, : \,i\in J, \,i\notin I, \, k\notin I\cup J \\ I\cap (J-\{i\})=\emptyset \\
    \mathrm{ord}(I,J-\{i\})=K}}  (-1)^{\varepsilon_i^J+\, \alpha(I,J-\{i\})} A_{I,i}^*\tt  A_{J,k}^*
    + 
      \sum_{\substack{I,J,\, j \, : \, j\in I, \,j\notin J, \, k\notin I\cup J \\ (I-\{j\})\cap J=\emptyset \\
    \mathrm{ord}(I-\{j\},J)=K}}  (-1)^{\varepsilon_j^I+|I|+\, \alpha(I-\{j\},J)} A_{I,k}^*\tt  A_{J,j}^*\nonumber\\
    + &
\sum_{\substack{I,J\, :\, I\cap J=\emptyset\\ \mathrm{ord}(I,J)=K}}\,  (-1)^{\alpha(I,J)}\, \Bigg\{
(-1)^{|J|} \Big[ \la\, (n-|J|+1-|I|) + \mu\, (1-|I|)\Big]
 (A_{I,k}^*\tt B_J^*)\nonumber\\
& \qquad -
    (-1)^{(|I|+1)|J|+|J|} \,
   \Big[ (-\la-\mu)\, (n-|J|+1-|I|) + \mu\, (1-|I|)\Big]
 (B_J^*\tt A_{I,k}^*)\Bigg\},\nonumber
  \end{align}
and evaluating $
\la=\p\otimes 1$ and $\mu= -\p\otimes 1 - 1\otimes \p $, we observe that the first  sum in (\ref{sss2}) is the first sum in (\ref{S2}).
 Using (\ref{alfa}), the second and third sums  in (\ref{sss2})  become the second sum  in (\ref{S2}).
If we replace ${
\la=\p\otimes 1,\ \mu= -\p\otimes 1 - 1\otimes \p }$ in the fourth sum  in (\ref{sss2}), we get:
\begin{align}
& \sum_{\substack{I,J\, :\, I\cap J=\emptyset\\ \mathrm{ord}(I,J)=K}}\,  (-1)^{\alpha(I,J)+|J|}\, \Bigg\{
\Big[ (n-|J|)\, (\p\otimes 1) - (1-|I|)\, (1\otimes \p) \Big]
 (A_{I,k}^*\tt B_J^*)\nonumber\\
& \qquad \qquad -
    (-1)^{(|I|+1)|J|} \,
   \Big[ (n-|J|)\,(1\otimes \p) - (1-|I|)\, (\p\otimes 1)\Big]
 (B_J^*\tt A_{I,k}^*)\Bigg\},\nonumber
  \end{align}
obtaining the last sum in (\ref{S2}). Now, using (\ref{eq:De}), for  all $K$, with $K^c= \{i_1,\dots ,i_l\}$, and $i_1<\dots i_k<i_{k+1}< i_l$, where $K^c$ is  the complement of $K$ in $\{1,\dots,n\}$, we have
\begin{align}\label{sss3}
      \de &(A_{K,{i_k},{i_{k+1}}}^*) = 
    \nonumber \\
&\sum_{\substack{I,J,l \, : \, I\cap J=\emptyset \\
    \mathrm{ord}(I-\{l\},J)=K\\ l\in I; \,l\notin J; \, i_k,i_{k+1}\notin I\cup J}}  
    (-1)^{\varepsilon_l^I+1+|I|+|J|+\alpha(I-\{l\},J)}\,
\Big[ A_{I,{i_k},{i_{k+1}}}^*\tt  A_{J,l}^* - (-1)^{|I|(|J|+1)} A_{J,l}^*\tt A_{I,{i_k},{i_{k+1}}}^*\Big]
\nonumber\\
    + &
      \sum_{\substack{ I,J \, ,i,j\, : i\notin I, \,i\in J, \,
j\in I, \,j\notin J}}  
(-1)^{\varepsilon_{i}^J+\,\varepsilon_{j}^I+ |I|+|J|\, \alpha(I-\{j\},J-\{i\})}\, 
\Big(A_{K, i_k, {i_{k+1}}}^*\Big)_\mu \Big(A_{\mathrm{ord}(I-\{j\},J-\{i\}),j,i}\Big)\, 
\\
& \qquad \qquad \qquad \qquad \qquad \qquad
\times \Big[A_{I,i }^* \tt  A_{J,j}^*- (-1)^{(|I|+1)(|J|+1)}  A_{J,j}^*\tt A_{I,i }^*
    \Big]\nonumber\\   
+ &
      \sum_{\substack{ I,J \, ,i\, :  \,
i \notin I, i\in J,  \, I\cap J =\emptyset 
}}
(A^*_{K,i_k,i_{k+1}})_\mu \Big(\p   \left(\textstyle{\frac{|I|-1 }{|I|+|J|-n-1}}\right)    (-1)^{\varepsilon_i^J+\al(I,J-\{i\})}  
 \sum_{j\notin  I\cup (J-\{i\})}  \ A_{\mathrm{ord}(I,J-\{i\}),j,i}\Big) \nonumber\\
& \qquad \qquad \qquad \qquad \qquad \qquad
 \times \, 
 \Big[
 A_{I,i}^*\tt  B_J^*- (-1)^{(|I|+1)\, |J|}B_J^* \tt  A_{I,i}^*
 \Big]
\nonumber\\
+ &
      \sum_{\substack{ I,J \, ,i\, :  \,
i \notin I, i\in J,   \, I\cap J =\emptyset 
}}
(A^*_{K,i_k,i_{k+1}})_\mu \Big(    (-1)^{\varepsilon_i^J+\al(I,J-\{i\})}  
 \sum_{j\notin  I\cup J}  \ A_{\mathrm{ord}(I,J-\{i\}),j,i} \Big) \nonumber\\
& \qquad \qquad \qquad \qquad \qquad \qquad
 \times \, 
 \Big[
 \la \, A_{I,i}^*\tt  B_J^*- (-1)^{(|I|+1)\, |J|}(-\la-\mu)B_J^* \tt  A_{I,i}^*
 \Big]
\nonumber\\
 + &
\sum_{\substack{I,J}}\,  
\sum_{\substack{r=1   }}^{s-1}
(A^*_{K,i_k,i_{k+1}})_\mu ([A_{I,j_r,j_{r+1}}\,_\la B_J]) 
  \Big[  A_{I,j_r,j_{r+1}}^*\tt B_J^*  -
    (-1)^{|I||J|} \,
 B_J^*\tt A_{I,j_r,j_{r+1}}^*\Big],\nonumber
  \end{align}
where in the last sum we use the notation $ I^c=\{j_1,\dots,j_s\}$, with $j_1<\cdots <j_s$, where $I^c$ is the complement of $I$ in $\{1,\dots,n\}$. Observe that the first    sum in (\ref{sss3}) is the first   sum in (\ref{3S}). The second    sum in (\ref{sss3}) produce  the second   sum in (\ref{3S}).
Now, observe that the double sums in the third and fourth terms in (\ref{sss3}), can be joined, since the term $j=i$ is zero in the third term. Therefore, we have
\begin{align}\label{previo}
&  \sum_{\substack{ I,J \, ,i\, :  \,
i \notin I, i\in J,   \, I\cap J =\emptyset \\ 
\mathrm{ord}(I,J-\{i\})=K
}}
(-1)^{\varepsilon_i^J+\al(I,J-\{i\})} 
\sum_{j\notin  I\cup J} 
(A^*_{K,i_k,i_{k+1}})_\mu 
 \Big(  A_{\mathrm{ord}(I,J-\{i\}),j,i}\Big)
\\
& \qquad \qquad \qquad \qquad \qquad \qquad
 \times \, \Big[p(\la,\mu)\, A_{I,i}^*\tt  B_J^*- (-1)^{(|I|+1)\, |J|}\,p(-\la-\mu,\mu)\, B_J^* \tt  A_{I,i}^*\Big],\nonumber
\end{align}
where $p(\la,\mu)=\mu   \left(\textstyle{\frac{|I|-1 }{|I|+|J|-n-1}}\right) \, + \, \la $, and replacing ${
\la=\p\otimes 1,\ \mu= -\p\otimes 1 - 1\otimes \p }$ in  (\ref{previo}), we get:
\begin{align*}
 &  \sum_{\substack{ I,J \, ,i\, :  \,
i \notin I, i\in J,   \, I\cap J =\emptyset \\ 
\mathrm{ord}(I,J-\{i\})=K
}} \,\,
\sum_{j\notin  I\cup J} 
\left(\textstyle{\frac{(-1)^{ \varepsilon_{i}^J+ \, \alpha(I,J-\{i\})}}{|I|+|J|-n-1}}\right)\,
(A^*_{K,i_k,i_{k+1}})_\mu 
 ( A_{\mathrm{ord}(I,J-\{i\}),j,i})
 \nonumber\\
    &\qquad \times
\Bigg\{ (1-|I|)\,
\Big[ A_{I,i}^*\tt \p (B_J^*)
        - (-1)^{(|I|+1)|J|}(\p (B_J^*)\tt A_{I,i}^*)\Big]\,
\nonumber\\
& \qquad \qquad \qquad \qquad
+ (|J|-n)\,
\Big[ \p(A_{I,i}^*)\tt B_J^*
        - (-1)^{(|I|+1)|J|}( B_J^* \tt\p(A_{I,i}^*))\Big]
    \Bigg\}\nonumber
\end{align*}
since
\begin{align*}
   \frac{(1-|I|)(\p\otimes 1 + 1\otimes \p ) +(|I|+|J|-n-1)(\p\otimes 1) }{|I|+|J|-n-1}=
 \frac{(1-|I|)( 1\otimes \p ) +(|J|-n)(\p\otimes 1) }{|I|+|J|-n-1},
\end{align*}
\noindent 
obtaining the third  sum in (\ref{3S}). 

Finally, in the fifth  sum  in (\ref{sss3}), we have to consider the three cases given by the formulas of the brackets $[A_{I,j_r,j_{r+1}}\,_\la B_J]$. 
 The first case gives us the following sum
\begin{align}\label{FL}
&
\sum_{\substack{I,J\, :\, I\cap J=\emptyset\\ \mathrm{ord}(I,J)=K}}\,  (-1)^{\alpha(I,J)}\, \Bigg\{
  \Big[ \la\, (n-|J|-|I|) - \mu\, |I|\Big]
 A_{I,{i_k},{i_{k+1}}}^*\tt B_J^* \\
& \qquad \qquad -
    (-1)^{|I||J|} \,
   \Big[ (-\la-\mu)\, (n-|J|-|I|) - \mu\, |I|\Big]
 B_J^*\tt A_{I,{i_k},{i_{k+1}}}^*\Bigg\},\nonumber
  \end{align}
and, using that
\begin{equation*}
     (n-|J|-|I|)(\p\otimes 1) -   |I|(-\p\otimes 1 - 1\otimes \p )=
(n-|J|)(\p\otimes 1) +   |I|(  1\otimes \p )
\end{equation*}
when we replace ${
\la=\p\otimes 1,\ \mu= -\p\otimes 1 - 1\otimes \p }$ in  (\ref{FL}), we get:
\begin{align*}
& \sum_{\substack{I,J\, :\, I\cap J=\emptyset\\ \mathrm{ord}(I,J)=K}}\,  (-1)^{\alpha(I,J)}\,
 \Bigg\{ \!
  (n-|J|)
\Big[ \p (A_{I,{i_k},{i_{k+1}}}^*)\tt B_J^*
 -
    (-1)^{|I||J|}
B_J^*\tt \p (A_{I,{i_k},{i_{k+1}}}^*)
\Big]\nonumber\\
&  \qquad \qquad  \qquad \qquad\qquad + |I|
\Big[ A_{I,{i_k},{i_{k+1}}}^*\tt \p(B_J^*)
 -
    (-1)^{|I||J|}
\, \p(B_J^*)\tt A_{I,{i_k},{i_{k+1}}}^*
\Big]
\Bigg\},\nonumber
\end{align*}
obtaining the fourth sum in (\ref{3S}).

The second case of the brackets $[A_{I,j_r,j_{r+1}}\,_\la B_J]$ corresponds to $j_r\in J,  j_{r+1}\notin J $, and 
 gives us the following sum
 \begin{align}\label{FI}
 &
\sum_{\substack{I,J\, :\, I\cap J=\emptyset\\ \mathrm{ord}(I,J)=K}}\,  
\sum_{\substack{r\,:\, j_r\in J,  j_{r+1}\notin J  }}  (-1)^{\al(I,J)}\,
(A^*_{K,i_k,i_{k+1}})_\mu  \Big( \sum_{j\notin I\cup J}
A_{\mathrm{ord}(I,J),j,{j_{r+1}}}\Big) \\
&\hskip 5cm \times \Big[  q(\la,\mu)\, A_{I,j_r,j_{r+1}}^*\tt B_J^*  -
    (-1)^{|I||J|} \,q(-\la-\mu,\mu)\,
 B_J^*\tt A_{I,j_r,j_{r+1}}^*\Big],\nonumber
  \end{align}
where we use the notation $ I^c=\{j_1,\dots,j_s\}$, with $j_1<\cdots <j_s$, where $I^c$ is the complement of $I$ in $\{1,\dots,n\}$, and 
 $q(\la,\mu)= \left(\la +\mu \, \textstyle{\frac{ |I|}{|I|+|J|-n}}\right)$. Therefore,   
 when we replace ${
\la=\p\otimes 1,\ \mu= -\p\otimes 1 - 1\otimes \p }$ in  (\ref{FI}), we get:
\begin{align*}
& \sum_{\substack{I,J\, :\, I\cap J=\emptyset\\ \mathrm{ord}(I,J)=K}}   \
\sum_{\substack{r\,:\, j_r\in J,  j_{r+1}\notin J  }}  \left(\textstyle{\frac{(-1)^{\al(I,J)}}{|I|+|J|-n}}\right)\,
(A^*_{K,i_k,i_{k+1}})_\mu  \Big( \sum_{j\notin I\cup J}
A_{\mathrm{ord}(I,J),j,{j_{r+1}}}\Big) \\
 & \qquad \times \Bigg\{ \!
  (|J|-n)
\Big[ \p (A_{I,{j_r},{j_{r+1}}}^*)\tt B_J^*
 -
    (-1)^{|I||J|}
B_J^*\tt \p (A_{I,{j_r},{j_{r+1}}}^*)
\Big]\nonumber\\
&  \qquad \qquad  \qquad \qquad\qquad - |I|
\Big[ A_{I,{j_r},{j_{r+1}}}^*\tt \p(B_J^*)
 -
    (-1)^{|I||J|}
\, \p(B_J^*)\tt A_{I,{j_r},{j_{r+1}}}^*
\Big]
\Bigg\},\nonumber
\end{align*}
obtaining the fifth sum in (\ref{3S}).
 
The third case of the brackets $[A_{I,j_r,j_{r+1}}\,_\la B_J]$ corresponds to $j_r\notin J,  j_{r+1}\in J $, and 
 gives us the following sum
 \begin{align}\label{FE}
 &
-\sum_{\substack{I,J\, :\, I\cap J=\emptyset\\ \mathrm{ord}(I,J)=K}}\,  
\sum_{\substack{r\,:\, j_r\notin J,  j_{r+1}\in J  }}  (-1)^{\al(I,J)}\,
(A^*_{K,i_k,i_{k+1}})_\mu  \Big( \sum_{j\notin I\cup J}
A_{\mathrm{ord}(I,J),j,{j_{r}}}\Big) \\
&\hskip 5cm \times \Big[  q(\la,\mu)\, A_{I,j_r,j_{r+1}}^*\tt B_J^*  -
    (-1)^{|I||J|} \,q(-\la-\mu,\mu)\,
 B_J^*\tt A_{I,j_r,j_{r+1}}^*\Big],\nonumber
  \end{align}
where we use the notation $ I^c=\{j_1,\dots,j_s\}$, with $j_1<\cdots <j_s$, where $I^c$ is the complement of $I$ in $\{1,\dots,n\}$, and 
 $q(\la,\mu)= \left(\la +\mu \, \textstyle{\frac{ |I|}{|I|+|J|-n}}\right)$. Therefore,   
 when we replace ${
\la=\p\otimes 1,\ \mu= -\p\otimes 1 - 1\otimes \p }$ in  (\ref{FE}), we get:
\begin{align*}
& -\sum_{\substack{I,J\, :\, I\cap J=\emptyset\\ \mathrm{ord}(I,J)=K}}   \
\sum_{\substack{r\,:\, j_r\notin J,  j_{r+1}\in J  }}  \left(\textstyle{\frac{(-1)^{\al(I,J)}}{|I|+|J|-n}}\right)\,
(A^*_{K,i_k,i_{k+1}})_\mu  \Big( \sum_{j\notin I\cup J}
A_{\mathrm{ord}(I,J),j,{j_{r}}}\Big) \\
 & \qquad \times \Bigg\{ \!
  (|J|-n)
\Big[ \p (A_{I,{j_r},{j_{r+1}}}^*)\tt B_J^*
 -
    (-1)^{|I||J|}
B_J^*\tt \p (A_{I,{j_r},{j_{r+1}}}^*)
\Big]\nonumber\\
&  \qquad \qquad  \qquad \qquad\qquad - |I|
\Big[ A_{I,{j_r},{j_{r+1}}}^*\tt \p(B_J^*)
 -
    (-1)^{|I||J|}
\, \p(B_J^*)\tt A_{I,{j_r},{j_{r+1}}}^*
\Big]
\Bigg\},\nonumber
\end{align*}
obtaining the last sum in (\ref{3S}).

\end{proof}

\

\subsection{Differential Lie   supercoalgebra $(K_n)^{*_c}$}\lbb{Kn}

Recall from (\ref{eq:k1}) that $K_n=\CC[\p]\otimes \Lambda(n)$. As usual, we denote by $\xi_I^*$  the elements in the dual basis of $\Lambda(n)$.  In order to apply (\ref{eq:De}), we need to compute the $\la\,$-brackets (\ref{eq:k2}) in $K_n$:
\begin{equation*}
[{\xi_I}_\la \xi_J]=\bigg( (|I|-2)\p (\xi_I\xi_J) + (-1)^{|I|} \sum_{i=1}^n (\p_i \xi_I)(\p_i
\xi_J)\bigg) + \lambda \, (|I|+|J|-4)\, \xi_I \xi_J.
\end{equation*}
Now, using the notations introduced in (\ref{epsi}) and (\ref{WW1}), we can write the co-product in $(K_n)^\c$.

\begin{theorem}
  The Lie supercoalgebra structure on $(K_n)^\c$ is defined by the following formula, in the dual basis:

  \begin{align}\label{K1}
    \de (\xi_K^*) &= \sum_{\substack{I,J \, : \, I\cap J=\emptyset \\
    \mathrm{ord}(I,J)=K}}  (-1)^{\alpha(I,J)}\, (|J|-2)\,  \Big[\, \d \, \xi_I^*\tt \xi_J^*-
    (-1)^{|I|\,|J|} \,
\xi_J^*\tt \d \, \xi_I^*\Big] \\
     &
     \ \ \ \ +   \sum_{\substack{i=1\dots n \\ I,J \, : \, \{i\}=I\cap J \\
    \mathrm{ord}(I-\{i\},J-\{i\})=K}}  (-1)^{|I|+\varepsilon_i^J+\varepsilon_j^I+\alpha(I-\{i\},J-\{i\})}\, \xi_I^*\tt \xi_J^*.
     \nonumber
  \end{align}
\end{theorem}

\vskip .1cm

\begin{proof} We apply (\ref{eq:De}), in order to compute the co-product:
  \begin{align*}
    \de (\xi_K^*)
     &=\bigg\{\sum_{I,J} (\xi_K^*)_\mu \big([{\xi_I}_\la \xi_J]\big) (\xi_I^*\tt \xi_J^*)
        \bigg\}_{|_{\,_{\substack{\la=\p\otimes 1\qquad\ \\ \mu= -\p\otimes 1 - 1\otimes \p}}}}
          \\
     &= \Bigg\{\sum_{\substack{I,J \, : \, I\cap J=\emptyset }}  (-1)^{\alpha(I,J)}\, (|I|-2)\, (\xi_K^*)_\mu \big(\p \,\xi_{\mathrm{ord}(I,J)}\big)
    ( \xi_I^*\tt \xi_J^*) \\
     &
     +   \sum_{\substack{i=1\dots n \\ I,J \, : \, \{i\}=I\cap J }}  (-1)^{|I|+\varepsilon_i^J+\varepsilon_j^I}\,(\xi_K^*)_\mu \big(\xi_{ I-\{i\}}\xi_{J-\{i\}}\big) (\xi_I^*\tt \xi_J^*)\\
     &+\,\la\, \sum_{\substack{I,J \, : \, I\cap J=\emptyset}}  (|I|+|J|-4) \, (-1)^{\alpha(I,J)}\,  (\xi_K^*)_\mu \big( \,\xi_{\mathrm{ord}(I,J)}\big)
    ( \xi_I^*\tt \xi_J^*)\Bigg\}_{|_{\,_{\substack{\la=\p\otimes 1\qquad\ \\ \mu= -\p\otimes 1 - 1\otimes \p}}}}
     \nonumber
  \end{align*}
\vskip .2cm

\noindent Joining the first and third sums and using that
\begin{equation*}
  \Big[\, \mu\, (|I|-2) \, + \, \la\, (|I|+|J|-4)\, \Big]|_{\,_{\substack{\la=\p\otimes 1\qquad\ \\ \mu= -\p\otimes 1 - 1\otimes \p}}}=(|J|-2) (\p\tt 1)-(|I|-2) (1\tt\p),
\end{equation*}
we obtain
\begin{align*}
    \de (\xi_K^*) &= \sum_{\substack{I,J \, : \, I\cap J=\emptyset \\
    \mathrm{ord}(I,J)=K}}  (-1)^{\alpha(I,J)}\, \Big[\, (|J|-2)\,  (\, \d \, \xi_I^*\tt \xi_J^*) - (|I|-2) \,
(\xi_I^*\tt \d \, \xi_J^*)\,\Big] \\
     &
     +   \sum_{\substack{i=1\dots n \\ I,J \, : \, \{i\}=I\cap J \\
    \mathrm{ord}(I-\{i\},J-\{i\})=K}}  (-1)^{|I|+\varepsilon_i^J+\varepsilon_j^I+\alpha(I-\{i\},J-\{i\})}\, \xi_I^*\tt \xi_J^*.
     \nonumber
  \end{align*}
Now, interchanging the sum over $I$ and $J$ in the first sum, and using (\ref{alfa}), we obtain (\ref{K1}), finishing the proof.
\end{proof}

\subsection{Differential Lie   supercoalgebra $(K_4')^{*_c}$}\lbb{K4}

We need the following general notation:
\begin{equation}\label{not}
 \xi_\star:=\xi_1\dots\xi_n.
\end{equation}
In general, the element $\xi_{I^c}$ will be considered with the set $I^c$ already ordered, where $I^c$ is the complement of $I$ in $\{1,\dots,n\}$.

Recall that the derived algebra $K_4'$ is simple
and has codimension 1 in $K_4$. More precisely,
\begin{equation*}
  K_4 = K_4' \oplus \CC \, \xi_\star.
\end{equation*}
Therefore,  a $\cp$-basis of $K_4'$ is given by
\begin{equation*}
  \{\xi_I\, :\, |I|\leq 3\}\cup \{\p \xi_\star\},
\end{equation*}
and the remaining $\la$-brackets are $[\,{\p \xi_\star\,}_\la \,\p \xi_\star\,]=0$ and
\begin{equation}\label{equa.b}
[\,{\p \xi_\star\,}_\la \,\xi_J\,]=
\left\{
  \begin{array}{ll}
     - 2\,\la\, \p \xi_\star, & \hbox{ if $|J|=0$;} \\
    - \la \,  \p_j \, \xi_\star , & \hbox{ if $J=\{j\}$;} \\
    \qquad 0\quad , & \hbox{ if $|J|=2,3$.}
  \end{array}
\right.
\end{equation}

\noindent Now, using the notations introduced in (\ref{epsi}), we can write the co-product in $(K_4')^\c$.

\begin{theorem}
  The Lie supercoalgebra structure on $(K_4')^\c$ is defined by the following formulas, in the dual basis:
\begin{equation}\label{K4'-2}
  \de((\p\xi_\star)^*)  = - 2 \, \big( \p \,(\p \xi_\star)^*\tt 1^*
- 1^*\tt \p \,(\p \xi_\star)^*\big)
- \sum_{i=1}^4\, (-1)^{\varepsilon_i^{\{i\}^c}}\, \big(\xi_{\{i\}^c}^*\tt \xi_i^* + \xi_i^*\tt \xi_{\{i\}^c}^* \big)
,
\end{equation}

\noindent and
  \begin{align}\label{K4'}
    \de (\xi_K^*)\, &= \, \delta_{K_4}(\xi_K^*)\, + \, \tilde\delta_{|K|,3}\  (-1)^{\varepsilon^{\{1,2,3,4\}}_{K^c}} \Big[\xi_{K^c}^*\tt \p \,(\p \xi_\star)^*
-\p \,(\p \xi_\star)^*\tt \xi_{K^c}^*\Big],
  \end{align}
\vskip .2cm

\noindent where $\tilde\delta_{|K|,3}$ is the Kronecker delta and $\delta_{K_4}(\xi_K^*)$ corresponds to the coproduct in $(K_4)^{*_c}$ given by the equations (\ref{eq.k0}), (\ref{eq.k1}),(\ref{eq.k2}) and (\ref{444}), where in (\ref{444}) we have to remove the terms with $\xi_\star^*$  (see Subsection \ref{K2}).
\end{theorem}

\begin{proof} We apply (\ref{eq:De}), in order to compute the co-product. For any $f\in (K_4')^{*_c}$, we have:
  \begin{align}\label{equa.a}
    \de (f)
     &=\left[\sum_{ |I|,|J|\leq 3} f_\mu \big([{\xi_I}_\la \xi_J]\big) (\xi_I^*\tt \xi_J^*)
     + \sum_{|J|\leq 3} f_\mu \big([{\p \xi_\star}_\la \xi_J]\big) ((\p \xi_\star)^*\tt \xi_J^*) \right.\nonumber\\
     &\qquad \left.
     + \sum_{|J|\leq 3} f_\mu \big([{\xi_J}_\la\p \xi_\star]\big) (\xi_J^*\tt (\p \xi_\star)^*) \right] |_{\substack{\la=\p\otimes 1\qquad\ \\ \mu= -\p\otimes 1 - 1\otimes \p}}
\end{align}

\noindent First, if $f=\xi_K^*$ with $|K|\leq 3$, then the first sum in (\ref{equa.a}) corresponds to $\delta_{K_4}(\xi_K^*)$, the coproduct in $(K_4)^{*_c}$ given by the equations (\ref{eq.k0}), (\ref{eq.k1}),(\ref{eq.k2}) and (\ref{444}), where in (\ref{444}) we have to remove the terms with $\xi_\star^*$  (see Subsection \ref{K2}). Now, replacing the $\la$-brackets (\ref{equa.b}) into the last two sums in (\ref{equa.a}), we obtain (\ref{K4'}).

Second, if $f=(\p \xi_\star)^*$, then in the first sum of (\ref{equa.a}), we only have to consider the cases $|I|=1, \,|J|=3$ or $|I|=3, \,|J|=1$. These two cases produce the last sum in (\ref{K4'-2}). Finally,
 using (\ref{equa.b}), we only have to consider $|J|=0$  in the other two sums in (\ref{equa.a}),  obtaining the first term in (\ref{K4'-2}), finishing the proof.
\end{proof}

\subsection{Differential Lie   supercoalgebra $(CK_6)^{*_c}$}\lbb{CK6}

We need the following notation:
\begin{equation*}
  \xi_{i_1 \dots i_r}:=\xi_{i_1}\dots \xi_{i_r}
\end{equation*}
and for  a monomial $f=\xi_{i_1} \dots \xi_{i_r}\in \La (n)$, we let ${f}^{\, \bullet}$ be its
Hodge dual, i.e. the unique monomial in $\La (n)$ (with a sign) such that
${f} f^{\, \bullet}= \xi_1 \dots \xi_n$. It is easy to see that
\begin{equation}\label{bull}
\xi_I^\bullet=(-1)^{\alpha(I,I^c)}\, \xi_{I^c},
\end{equation}
where $I^c$ is ordered, and $I^c$ is the complement of $I$ in $\{1,\dots,n\}$.

Recall that the Lie conformal superalgebra $CK_6$ is defined as the subalgebra
of $K_6$ given by
\begin{equation*}
    CK_6=\cp \hbox{-span}\  \{f+\beta(-1)^{\frac{|f|(|f|+1)}{2}}(-\p)^{3-|f|}
    \, f^{\, \bullet}\, : \, f\in \La (6),\  0\leq |f|\leq 3\}.
\end{equation*}
\vskip .1cm

\noindent where $\beta^2=-1$.

Therefore, we consider the following generators:
\begin{align*}
  L \,\,&= -\tfrac{1}{2} (1-\beta\p^3\xi_\star) \\
  C_i \,\,&= \,\ \xi_i -\be \, \p^2 \xi_i^\bullet,  \\
  C_{ij}\, &= \,\xi_{ij} + \be \, \p \, \xi_{ij}^\bullet, \\
  C_{ijk} &= \,\xi_{ijk} + \be \,  \xi_{ijk}^\bullet.
\end{align*}

\vskip .2cm

\

The $\la$-brackets are given by (cf. \cite{CK2}):
\begin{align}\label{CK6-b1}
  [L_\la L]\,\, =& \, \,  (2\la +\p)\, L \qquad\qquad
[L_\la C_{ij}]\, = \,\,  (\tfrac{3}{2}\la +\p)\, C_{ij} \\
\label{CK6-b11}
  [L_\la C_i] \,\,=& \,\,  (\la +\p)\, C_i \qquad\qquad
  [L_\la C_{ijk}] = \, \, (\tfrac{1}{2}\la +\p)\, C_{ijk},
\end{align}
together with
\begin{align}\label{CK6-b21}
  [\,{C_i\,}_\la C_j\,]\,\, =& \, \,  -(2\la +\p)\, C_{ij} \, +\, 2 \, \tilde{\delta}_{i,j}\, L \\
[\,{C_{ij}\,}_\la C_k\,] \,\,=& \,\,  -\la \, C_{ijk}\, +\, \tilde{\delta}_{k,i} \, C_j \, -\, \tilde{\delta}_{k,j} \, C_i , \label{CK6-b22}
\end{align}
and
\begin{align}\label{CK6-b31}
  [\,{C_{ij}\,}_\la C_{kl}\,]\,  \, =& \ \  \tilde{\delta}_{i,k} \, C_{jl} \, -\, \tilde{\delta}_{i,l} \, C_{jk} 
   -\tilde{\delta}_{j,k} \, C_{il} \, + \, \tilde{\delta}_{j,l} \, C_{ik}\\
\label{CK6-b32}
  [\,{C_{ij}\,}_\la C_{klm}\,]\, \, =& \ \  \tilde{\delta}_{i,k} \, C_{jlm} \, -\, \tilde{\delta}_{i,l} \, C_{jkm}
\, +\, \tilde{\delta}_{i,m} \, C_{jkl} \, -\, \tilde{\delta}_{j,k} \, C_{ilm}
\, +\, \tilde{\delta}_{j,l} \, C_{ikm} \, -\, \tilde{\delta}_{j,m} \, C_{ikl}
\\
  [\,{C_{ijk}\,}_\la C_{lmn}\,]\, \, =& \ \ 0.\nonumber
\end{align}
Finally
\begin{equation}\label{CK6-b4}
  [\,{C_{i}\,}_\la C_{jkl}\,]\, \, = \ \ \beta \, (\xi_{ijkl}^\bullet + \beta\,\p \, \xi_{ijkl})\, -\, \left((\xi_i(\xi_{jkl}^\bullet))^\bullet + \beta \,\p\, \xi_i(\xi_{jkl}^\bullet)\right),
\end{equation}
where $\tilde{\de}$ is the Kronecker delta.
\

We have that the set 
\begin{equation*}
  \{L\}\,\cup\,\{C_i \,|\, i=1\dots 6\}\,\cup\,\{C_{ij}\,|\,i<j\}\,\cup\,\{C_{1jk}\,|\,1<j<k\}
\end{equation*}
form a basis of
$CK_6$ as a $\cp$-module.

\begin{theorem}
  The Lie supercoalgebra structure on $(CK_6)^\c$ is defined by the following formulas, in the dual basis:
\begin{align}\label{CK6-1}
  \de(L^*)  &=   \, \p L^*\tt L^*
- L^*\tt \p L^* \, +\,
 \sum_{i=1}^6 \, 2\, \big(C_i^*\tt C_i^* \big)
,\\
\de(C_l^*)  &=   \,  \p\, C_l^*\tt L^*
- L^*\tt \p \, C_l^* \, +\,
 \sum_{k<l} \,  \big(C_{kl}^*\tt C_k^* - C_k^*\tt C_{kl}^* \big) +
 \sum_{l<k} \,  \big(C_{lk}^*\tt C_k^* - C_k^*\tt C_{lk}^* \big).
\nonumber
\end{align}

\noindent For $1<s<t$, we consider $I:=\{1,s,t\}$ and $\{a,b,c\}:=I^c$, with $a<b<c$, where $I^c$ is the complement of $I$ in $\{1,\dots,6\}$. Then, we have:

  \begin{align}\label{CK6-2}
 \de(&C_{1st}^*)=
\big(\p C_{1st}^*\tt L^* - L^*\tt \p  C_{1st}^* \big)  +
 \tfrac{1}{2}\big(C_{1st}^*\tt \p L^*- \p L^*\tt  \, C_{1st}^*\big) \\
&
 - \!
\big(\p C_{rs}^*\tt  C_t^*- C_t^*\tt  \p  C_{rs}^*\big)
\! + \!
\big(\p C_{rt}^*\tt  C_s^*- C_s^*\tt  \p  C_{rt}^*\big)
\! - \!
\big(\p C_{st}^*\tt  C_r^*- C_r^*\tt  \p  C_{st}^*\big)\quad \nonumber\\
 &
   \hbox{\small $
  \,  +\be \,(-1)^{\al(I^c,I)} \Big[
        (C_{1a}^* \tt C_{1bc}^* - C_{1bc}^*\tt  C_{1a}^*) +
        (C_{1b}^* \tt C_{1ac}^* - C_{1ac}^*\tt  C_{1b}^*) +
        (C_{1c}^* \tt C_{1ab}^* - C_{1ab}^*\tt  C_{1c}^*)\Big]$}
\nonumber\\
&     + \sum_{1<i<s}\, \Big[
        (C_{is}^* \tt C_{1it}^* - C_{1it}^*\tt  C_{is}^*) -
        (C_{it}^* \tt C_{1is}^* - C_{1is}^*\tt  C_{it}^*)\Big]
\nonumber
\\
&    + \sum_{s<i<t} \! \Big[
         -\!
        (C_{si}^* \tt C_{1it}^* \!-\! C_{1it}^*\tt  C_{si}^*)
 \! +\!
        (C_{it}^* \tt C_{1si}^* \!-\! C_{1si}^*\tt  C_{it}^*)\Big]
\nonumber
\\
&   +
 \sum_{t<i}\,
        \Big[
        (C_{si}^* \tt C_{1ti}^* - C_{1ti}^*\tt  C_{si}^*)
           - (C_{ti}^* \tt C_{1si}^* - C_{1si}^*\tt  C_{ti}^*)
        \Big].
        \nonumber
\end{align}

\vskip .2cm

\noindent And for $r<s$:

  \begin{align}\label{CK6-3}
\de(C_{rs}^*)  &=   \,  \big( \p C_{rs}^*\tt L^*
- L^*\tt \p  C_{rs}^*\big) \, -
\, \tfrac{1}{2}\big(C_{rs}^*\tt  \,\p L^*- \p L^* \tt C_{rs}^*\big)
+
\big( C_{r}^*\tt \p C_s^*+ \p C_s^*\tt    C_{r}^*\big)
\nonumber
\\
& \ \ \ \ -
\big(\p  C_{r}^*\tt  C_s^*+  C_s^*\tt  \p  C_{r}^*\big)
 + \sum_{i<r} \,  \big(C_{ir}^*\tt C_{is}^* - C_{is}^*\tt C_{ir}^* \big)
\nonumber
\\
& \ \ \ \
-
 \sum_{r<i<s}   \big(C_{ri}^*\tt C_{is}^* - C_{is}^*\tt C_{ri}^* \big)
+ \sum_{s<i} \,  \big(C_{ri}^*\tt C_{si}^* - C_{si}^*\tt C_{ri}^* \big) \\
& \ \ \ \  +
\sum_{i=1}^{6}\sum_{\substack{1<k<l\\ \{i,1,k,l\}^c=\{r,s\}}}
 \beta\, (-1)^{\alpha(\{i\},\{1,k,l\})+\alpha(\mathrm{ord}(\{i\},\{1,k,l\}),\{r,s\})}
\big(C_{i}^*\tt C_{1kl}^* + C_{1kl}^*\tt C_{i}^* \big) \nonumber\\
& \ \ \ \
- \tilde{\delta}_{1<r} \, (-1)^A\,  \big(C_{1}^*\tt C_{1rs}^* + C_{1rs}^*\tt C_{1}^* \big)
\nonumber
\\
& \ \ \ \ 
-
 \tilde{\delta}_{1,r}\, \Big[\sum_{1<i<s} (-1)^B\, \big(C_{i}^*\tt C_{1is}^* + C_{i1s}^*\tt C_{i}^* \big)
- \sum_{s<i} (-1)^C\,  \big(C_{i}^*\tt C_{1si}^* + C_{1si}^*\tt C_{i}^* \big)\Big]\nonumber
\end{align}
where
\begin{align*}
  A &= \al(\{1,r,s\},\{1,r,s\}^c)+\al(\{1\},\{1,r,s\}^c)+
        \al({\mathrm{ord}}(\{1\},\{1,r,s\}^c),\{r,s\}) \\
  B &= \al(\{1,i,s\},\{1,i,s\}^c)+\al(\{i\},\{1,i,s\}^c)+
        \al({\mathrm{ord}}(\{i\},\{1,i,s\}^c),\{1,s\})  \\
  C &= \al(\{1,s,i\},\{1,s,i\}^c)+\al(\{i\},\{1,s,i\}^c)+
        \al({\mathrm{ord}}(\{i\},\{1,s,i\}^c),\{1,s\}) ,
\end{align*}
and we consider $I^c$ an ordered set, where $I^c$ is the complement of $I$ in $\{1,\dots,6\}$.

\end{theorem}

\begin{proof}
   Using as before the equation (\ref{eq:De}), it is easy to see that all the terms with $L^*$ in (\ref{CK6-1}), (\ref{CK6-2}) and (\ref{CK6-3}), are obtained by the brackets in (\ref{CK6-b1}) and (\ref{CK6-b11}). Similarly, the sums in the expressions of $\de(L^*)$ and $\de(C_l^*)$ follow by the brackets in (\ref{CK6-b21}) and (\ref{CK6-b22}), by taking care with the Kronecker delta.

The third, fourth and fifth terms in (\ref{CK6-2}) are obtained by the coefficient in $\la$ in (\ref{CK6-b22}). The other sums in (\ref{CK6-2}) follow by a careful analysis of all the Kronecker delta in (\ref{CK6-b32}). We shall explain it in detail. We have that
the corresponding part of (\ref{eq:De}) is the following:
\begin{equation}\label{CK6-15}
   \sum_{\substack{i<j \\ 1<l<m}} \Big[
 (C_{1st}^*)_\mu\big([{C_{ij}}_\la C_{1lm}]\big)
\big(C_{ij}^*\tt C_{1lm}^*\big) + (C_{1st}^*)_\mu\big([{C_{1lm}}_\la{C_{ij}} ]\big)\big( C_{1lm}^*\tt C_{ij}^* \big) \Big]
 |_{\substack{\la=\p\otimes 1\qquad\ \\ \mu= -\p\otimes 1 - 1\otimes \p}}
\end{equation}
Observe that the $\la$-brackets in (\ref{CK6-15}) correspond to  (\ref{CK6-b32}). We have to consider the six terms in (\ref{CK6-b32}). The first one produce in (\ref{CK6-15}) the following sum:
\begin{equation}\label{CK6-25}
   \sum_{\substack{i<j \\ 1<l<m}} \Big[
 (C_{1st}^*)_\mu\big(\tilde{\de}_{i,1}\, C_{jlm}\big)
\big(C_{ij}^*\tt C_{1lm}^*\big)
-
(C_{1st}^*)_\mu\big(\tilde{\de}_{i,1}\, C_{jlm}\big)\big( C_{1lm}^*\tt C_{ij}^* \big) \Big]
 |_{\substack{\la=\p\otimes 1\qquad\ \\ \mu= -\p\otimes 1 - 1\otimes \p}}
\end{equation}
and we have to sum over the cases where $(C_{1st}^*)_\mu\big(
 C_{jlm}\big)\neq 0$. Observe that $1<j$ and $1<l<m$, therefore 
  we have that $\{1,s,t\}=\{j,l,m\}^c$. Using the notation $\{1,s,t\}^c=\{a,b,c\}$, with $a<b<c$, and the condition $l<m$, 
 we only have  the following three cases: $(j=a,l=b,m=c)$, or  $(l=a,j=b,m=c)$, or $(l=a,m=b,j=c)$. Hence, we need to compute $(C_{1st}^*)_\mu\big(
 C_{\{1,s,t\}^c}\big)$. Observe that, using (\ref{alfa}) and (\ref{bull}), for any $I\subset \{1,\dots,6\}$ with $|I|=3$, we have
 \begin{align*}
   C_{I^c} &= \xi_{I^c}+ \beta\, \xi_{I^c}^\bullet=(-1)^{\al(I,I^c)} \xi_I^\bullet+\be \,(-1)^{\al(I^c,I)} \xi_I= \be \,(-1)^{\al(I^c,I)} (\xi_I+ \be \xi_I^\bullet)=\be \,(-1)^{\al(I^c,I)} C_I.
 \end{align*}
Therefore, the sum (\ref{CK6-25}) is equal to
\begin{equation*}
  \be \,(-1)^{\al(I^c,I)} \Big[
        (C_{1a}^* \tt C_{1bc}^* - C_{1bc}^*\tt  C_{1a}^*) +
        (C_{1b}^* \tt C_{1ac}^* - C_{1ac}^*\tt  C_{1b}^*) +
        (C_{1c}^* \tt C_{1ab}^* - C_{1ab}^*\tt  C_{1c}^*)\Big],
\end{equation*}
with $I=\{1,s,t\}$, producing the terms in the third line in (\ref{CK6-2}).

Now, we have to consider the second term in (\ref{CK6-b32}). It will  produce in (\ref{CK6-15}) the following sum:
\begin{equation}\label{CK6-35}
   -\ \sum_{\substack{i<j \\ 1<l<m}} \Big[
 (C_{1st}^*)_\mu\big(\tilde{\delta}_{i,l} \, C_{j1m}\big)
\big(C_{ij}^*\tt C_{1lm}^*\big)
-
(C_{1st}^*)_\mu\big(\tilde{\delta}_{i,l} \, C_{j1m}\big)\big( C_{1lm}^*\tt C_{ij}^* \big) \Big]
 |_{\substack{\la=\p\otimes 1\qquad\ \\ \mu= -\p\otimes 1 - 1\otimes \p}}
\end{equation}
Again, we have to sum over the cases where $(C_{1st}^*)_\mu\big(
 C_{j1m}\big)\neq 0$, and $i=l$. Observe that $1<s<t$, therefore we have two cases: $(j=s,m=t)$, or $(m=s,j=t)$. In the first case,  we should take the sum over $i$  with $i<j=s$ and $1<l=i<m=t$, hence the sum over $1<i<s$, and in this case must change the sign since we are taking  $(C_{1st}^*)_\mu\big(
 C_{s1t}\big)$. In the second case,  we must take the sum over $i$ with $i<j=t$ and $1<l=i<m=s$, hence the sum over $1<i<s$. Therefore, the sum (\ref{CK6-35}) is equal to
\begin{equation*}
 - \sum_{1<i<s}\, \Big(
        -(C_{is}^* \tt C_{1it}^* - C_{1it}^*\tt  C_{is}^*) +
        (C_{it}^* \tt C_{1is}^* - C_{1is}^*\tt  C_{it}^*)\Big),
\end{equation*}
producing the  two terms in the fourth line in (\ref{CK6-2}).

Similarly, using the same ideas, we have to consider the four remaining terms in (\ref{CK6-b32}):
\begin{equation*}
  \tilde{\delta}_{i,m} \, C_{j1l} \, -\, \tilde{\delta}_{j,1} \, C_{ilm}
\, +\, \tilde{\delta}_{j,l} \, C_{i1m} \, -\, \tilde{\delta}_{j,m} \, C_{i1l}.
\end{equation*}
The term corresponding to $\tilde{\delta}_{i,m} \, C_{j1l}$ in (\ref{CK6-15}), gives the sum
\begin{equation*}
  \sum_{s<i<t}\,
        (C_{it}^* \tt C_{1si}^* - C_{1si}^*\tt  C_{it}^*) .
\end{equation*}
The term corresponding to $-\, \tilde{\delta}_{j,1} \, C_{ilm}$ in (\ref{CK6-15}), gives  the condition $i<j=1$, which is not possible.

\noindent
The term corresponding to $\tilde{\delta}_{j,l} \, C_{i1m}$ in (\ref{CK6-15}), gives the sum
\begin{equation*}
  - \sum_{s<i<t}\,
         (C_{si}^* \tt C_{1it}^* - C_{1it}^*\tt  C_{si}^*)
  .
\end{equation*}
The term corresponding to $-\, \tilde{\delta}_{j,m} \, C_{i1l}$ in (\ref{CK6-15}), gives  the sum
\begin{equation*}
  \sum_{t<i}\,
        \Big[
        (C_{si}^* \tt C_{1ti}^* - C_{1ti}^*\tt  C_{si}^*)
           - (C_{ti}^* \tt C_{1si}^* - C_{1si}^*\tt  C_{ti}^*)
        \Big].
\end{equation*}
Therefore, we obtained the remaining terms in (\ref{CK6-2}).

\

Now, it remains to prove (\ref{CK6-3}).
The third and fourth terms in (\ref{CK6-3}) are easily obtained using (\ref{CK6-b21}).
The next three sums in (\ref{CK6-3}) follow from (\ref{CK6-b31}).
The other terms in  (\ref{CK6-3}) correspond to the $\la$-brackets in (\ref{CK6-b4}). Namely, the corresponding part of (\ref{eq:De}) is the following:
\begin{equation}\label{CK6-5}
   \sum_{i=1}^{6}\sum_{\substack{1<k<l}} \Big[
 (C_{rs}^*)_\mu\big([{C_i}_\la C_{1kl}]\big)
\big(C_{i}^*\tt C_{1kl}^*\big) + (C_{rs}^*)_\mu\big([{C_{1kl}}_\la{C_i} ]\big)\big( C_{1kl}^*\tt C_{i}^* \big) \Big]
 |_{\substack{\la=\p\otimes 1\qquad\ \\ \mu= -\p\otimes 1 - 1\otimes \p}}
\end{equation}

\noindent Double sum   in (\ref{CK6-3}) will come from the first  term in (\ref{CK6-b4}). More precisely, we need to compute the part of (\ref{CK6-5}) corresponding to  the first term in (\ref{CK6-b4}). Observe  that in order to obtain a non-zero term after  applying $(C_{rs}^*)_\mu$ in (\ref{CK6-5}), we need to sum over $i,k$ and $l$ such that (here we use (\ref{bull}))
 \begin{equation}\label{CK6-7}
 \xi_{i1kl}^\bullet = (-1)^{\alpha(\{i\},\{1,k,l\})}\xi_{\mathrm{ord}(\{i\},\{1,k,l\})}^\bullet
                    =(-1)^{\alpha(\{i\},\{1,k,l\})
+\alpha(\mathrm{ord}(\{i\},\{1,k,l\}),\{r,s\})}\, \xi_{rs},
 \end{equation}
obtaining the double sum over $i,k$ and $l$ in (\ref{CK6-3}).

The last three terms in (\ref{CK6-3}) follow from the second term in (\ref{CK6-b4}). More precisely, we need to compute the part of (\ref{CK6-5}) corresponding to  the second term in (\ref{CK6-b4}). Observe that in order to obtain a non-zero term after  applying $(C_{rs}^*)_\mu$ in (\ref{CK6-5}), we need to sum over $i,k$ and $l$ such that
 \begin{equation}\label{CK6-6}
 (\xi_i(\xi_{1kl}^\bullet))^\bullet = (\mathrm{sgn})\, \xi_{rs}.
 \end{equation}
Therefore, $i$ must be equal to $1,k$ or $l$. Hence, we deduce that we have the following three cases: $(i=1,k=r,l=s)$, or $(i=k,1=r,l=s)$, or $(i=l,1=r,k=s)$. When we introduce each case in (\ref{CK6-5}), we clearly produce the last three sums in (\ref{CK6-3}). It remains to find the corresponding signs in (\ref{CK6-6}), for each case. Using the definition of the Hodge dual, we need to find $X$ such that $(\xi_i(\xi_{1kl}^\bullet)) \xi_{rs}=(-1)^X\, \xi_\star$. Using  (\ref{bull}), we have
\begin{align*}
  (\xi_i(\xi_{1kl}^\bullet)) \xi_{rs} &=
(-1)^{\al(\{1,k,l\},\{1,k,l\}^c)} (\xi_i \,\xi_{\{1,k,l\}^c})\, \xi_{rs} \\
   &= (-1)^{\al(\{1,k,l\},\{1,k,l\}^c)+\al(\{i\},\{1,k,l\}^c)}
        \xi_{\mathrm{ord}(\{i\},\{1,k,l\}^c)}\, \xi_{rs}\\
   &= (-1)^{\al(\{1,k,l\},\{1,k,l\}^c)+\al(\{i\},\{1,k,l\}^c)
           + \al({\mathrm{ord}}(\{i\},\{1,k,l\}^c),\{r,s\})}\, \xi_\star,
\end{align*}
where the complement of a set is ordered.
Therefore, we obtained the value of $X$, and using this,
 if we replace the cases in (\ref{CK6-5}), we get the signs $A,B$ and $C$ in (\ref{CK6-3}), respectively, finishing the proof.
\end{proof}

\

\subsection{N=2,3,4 differential Lie supercoalgebras}\lbb{K2}

We present the explicit description of  the Lie supercoalgebras associated to the OPE of the N=2,3,4 super conformal Lie algebras. They correspond to $K_2, K_3$ and $K_4$, respectively. Therefore, we shall use (\ref{K1}). In all these cases, we have
\begin{equation}\label{eq.k0}
  \delta(1^*)=2\,(1^*\tt \p \, 1^* - \p \, 1^*\tt 1^*) - \sum_{i=1}^{n} \xi_i^*\tt \xi_i^*.
\end{equation}

\noindent Recall the  notation:
\begin{equation}\label{nota}
  \xi_{i_1 \dots i_r}:=\xi_{i_1}\dots \xi_{i_r} \quad \hbox{ and }\quad \xi_\star:=\xi_1\dots\xi_n.
\end{equation}

\vskip .3cm

\noindent {\bf Case N=2:} Using that $K_2=\cp$-span$\{1,\xi_1,\xi_2,\xi_{12}\}$,   we have
\vskip -.1cm
\begin{equation*}
  \delta(\xi_i^*)=2(1^*\tt \p \xi_i^* - \p \xi_i^*\tt 1^*)
                    +(\xi_i^*\tt \p 1^* - \p 1^*\tt \xi_i^*)
                    +(\xi_{\{i\}^c}^* \tt \xi_{12}^* - \xi_{12}^* \tt \xi_{\{i\}^c}^*),
\end{equation*}
and
\begin{equation*}
  \delta(\xi_{12}^*)=2(1^*\tt \p \xi_{12}^* - \p \xi_{12}^*\tt 1^*)
                    -(\xi_2^*\tt \p \xi_1^* + \p \xi_1^*\tt \xi_2^*)
                    +(\xi_1^*\tt \p \xi_2^* + \p \xi_2^*\tt \xi_1^*).
\end{equation*}
\vskip .3cm

\noindent {\bf Case N=3:} for $i=1,2,3$, we have

\begin{equation*}
  \delta(\xi_i^*)=2(1^*\tt \p \xi_i^* - \p \xi_i^*\tt 1^*)
                    +(\xi_i^*\tt \p 1^* - \p 1^*\tt \xi_i^*)
                    +\sum_{k\neq i} (\xi_k^* \tt \xi_{ik}^* - \xi_{ik}^* \tt \xi_k^*),
\end{equation*}
\noindent together with (for $i<j$)
\begin{align*}
  \delta(\xi_{ij}^*)&=2(1^*\tt \p \xi_{ij}^* - \p \xi_{ij}^*\tt 1^*)
                    +(\xi_i^*\tt \p \xi_j^* + \p \xi_j^*\tt \xi_i^*)
                    -(\p\xi_i^*\tt \xi_j^* + \xi_j^*\tt \p\xi_i^*) \\
                    &\quad +(-1)^k\, (\xi_{123}^* \tt \xi_{k}^* + \xi_{k}^* \tt \xi_{123}^*)+(\xi_{kj}^* \tt \xi_{ik}^* - \xi_{ik}^* \tt \xi_{kj}^*)
\end{align*}
where $k\in\{i,j\}^c$, and
\begin{align*}
  \delta(\xi_{123}^*)&=2(1^*\tt \p \xi_{123}^* - \p \xi_{123}^*\tt 1^*)
                    +(\p 1^*\tt  \xi_{123}^* -  \xi_{123}^*\tt \p 1^*)\\
                    &\quad +(\xi_1^*\tt \p\xi_{23}^* - \p\xi_{23}^*\tt\xi_1^*)
                    -(\xi_2^*\tt \p\xi_{13}^* - \p\xi_{13}^*\tt\xi_2^*)
                    +(\xi_{3}^* \tt\p \xi_{12}^* - \p \xi_{12}^* \tt \xi_{3}^*).
\end{align*}

\vskip .5cm

\noindent {\bf Case N=4:} for $k=1,2,3,4$, we have
\vskip -.1cm

\begin{align}\label{eq.k1}
  \delta(\xi_k^*)= & \ 2\, (1^*\tt \p \xi_k^* - \p \xi_k^*\tt 1^*)
                    +(\xi_k^*\tt \p 1^* - \p 1^*\tt \xi_k^*)\\
                    & +\sum_{i < k} (\xi_{ik}^* \tt \xi_{i}^* - \xi_{i}^* \tt \xi_{ik}^*)+\sum_{k<i} (\xi_{i}^* \tt \xi_{ki}^* - \xi_{ki}^* \tt \xi_{i}^*).\nonumber
\end{align}
For $k<l$, we have
\vskip -.2cm

\begin{align}\label{eq.k2}
  \delta(\xi_{kl}^*) = &\, 2\, (1^*\tt \p \xi_{kl}^* - \p \xi_{kl}^*\tt 1^*)
                    +(\xi_k^*\tt \p \xi_l^* + \p \xi_l^*\tt \xi_k^*)
                    -(\p\xi_k^*\tt \xi_l^* + \xi_l^*\tt \p\xi_k^*)
\nonumber
\\
    &  + \ \sum_{i < k} \, \Big[(-\xi_{ikl}^* \tt \xi_{i}^* - \xi_{i}^* \tt \xi_{ikl}^*)+ (\xi_{ik}^* \tt \xi_{il}^* + \xi_{il}^* \tt \xi_{ik}^*)\Big] \nonumber
\\
&  +\sum_{k<i < l} \Big[(\xi_{kil}^* \tt \xi_{i}^* + \xi_{i}^* \tt \xi_{kil}^*)+ (-\xi_{ki}^* \tt \xi_{il}^* + \xi_{il}^* \tt \xi_{ki}^*)\Big] %
\\
&  +\ \sum_{l < i} \, \Big[(-\xi_{kli}^* \tt \xi_{i}^* - \xi_{i}^* \tt \xi_{kli}^*)+ (\xi_{ki}^* \tt \xi_{li}^* - \xi_{li}^* \tt \xi_{ki}^*)\Big] . \nonumber
\end{align}

\noindent For $k<l<m$, we have

\vskip -.2cm

\begin{align}\label{444}
& \delta(\xi_{klm}^*)= \, 2\, (1^*\tt \p \xi_{klm}^* - \p \xi_{klm}^*\tt 1^*)+(\p 1^*\tt  \xi_{klm}^* - \xi_{klm}^*\tt \p 1^*) \\
          & -( \p \xi_{kl}^*\tt \xi_m^* - \xi_m^*\tt\p \xi_{kl}^*)
            +(\p\xi_{km}^*\tt \xi_l^* - \xi_l^*\tt \p\xi_{km}^*)
            -( \p \xi_{lm}^*\tt \xi_k^* - \xi_k^*\tt\p \xi_{lm}^*) \nonumber
\\
    &  +  \sum_{i < k}  \hbox{\small $   \Big[ \xi_\star^*\tt \xi_i^*-\xi_i^*\tt \xi_\star^* -\xi_{ikl}^* \tt \xi_{im}^* + \xi_{im}^* \tt \xi_{ikl}^*
+ \xi_{ikm}^* \tt \xi_{il}^* - \xi_{il}^* \tt \xi_{ikm}^*
-\xi_{ilm}^* \tt \xi_{ik}^* + \xi_{ik}^* \tt \xi_{ilm}^*
\Big] $}\nonumber
\\
&   + \sum_{k< i < l} \! \! \hbox{\small $    \Big[\! - \! \xi_\star^*\tt \xi_i^*+\xi_i^*\tt \xi_\star^* +\xi_{kil}^* \tt \xi_{im}^* - \xi_{im}^* \tt \xi_{kil}^*- \xi_{kim}^* \tt \xi_{il}^* + \xi_{il}^* \tt \xi_{kim}^*
+\xi_{ilm}^* \tt \xi_{ki}^* - \xi_{ki}^* \tt \xi_{ilm}^*
\Big] $}\nonumber
\\
&   +  \sum_{l < i < m} \!\! \hbox{\small $ \Big[ \xi_\star^*\tt \xi_i^*
-\xi_i^*\tt \xi_\star^*
-\xi_{kli}^* \tt \xi_{im}^*
+ \xi_{im}^* \tt \xi_{kli}^*
+ \xi_{kim}^* \tt \xi_{li}^*
- \xi_{li}^* \tt \xi_{kim}^*
-\xi_{lim}^* \tt \xi_{ki}^*
+ \xi_{ki}^* \tt \xi_{lim}^*
\Big]  $}\nonumber
\\
&   +  \sum_{m < i} \! \hbox{\small $ \Big[
-\xi_\star^*\tt \xi_i^*
+\xi_i^*\tt \xi_\star^*
+\xi_{kli}^* \tt \xi_{mi}^*
- \xi_{mi}^* \tt \xi_{kli}^*
- \xi_{kmi}^* \tt \xi_{li}^*
+ \xi_{li}^* \tt \xi_{kmi}^*
+\xi_{lmi}^* \tt \xi_{ki}^*
- \xi_{ki}^* \tt \xi_{lmi}^*
\Big]  $}\nonumber
\end{align}
\vskip .2cm

\noindent and finally, we have

\begin{align*}
\delta(\xi_{\star}^*)= & \, 2\, (1^*\tt \p \xi_{\star}^*
- \p \xi_{\star}^*\tt 1^*)
+2(\p 1^*\tt  \xi_{\star}^* - \xi_{\star}^*\tt \p 1^*)\\
           &
          -( \p \xi_{123}^*\tt \xi_4^* + \xi_4^*\tt\p \xi_{123}^*)
          -( \xi_{123}^*\tt \p\xi_4^* + \p \xi_4^*\tt\xi_{123}^*) \\
           &
          +(  \p \xi_{124}^*\tt \xi_3^* + \xi_3^*\tt  \p \xi_{124}^*)
          +(  \xi_{124}^*\tt \p \xi_3^* + \p \xi_3^*\tt \xi_{124}^*) \\
           &
          -(  \p \xi_{134}^*\tt \xi_2^* + \xi_2^*\tt  \p \xi_{134}^*)
          -(  \xi_{134}^*\tt \p \xi_2^* + \p \xi_2^*\tt \xi_{134}^*)\\
           &
          +(  \p \xi_{234}^*\tt \xi_1^* + \xi_1^*\tt  \p \xi_{234}^*)
          +(  \xi_{234}^*\tt \p \xi_1^* + \p \xi_1^*\tt \xi_{234}^*)
.
%
\end{align*}

\

\section{classification of Jordan conformal superalgebras and differential Jordan supercoalgebras}\lbb{classification-Jordan}

\

In this section we present the definitions of Jordan conformal superalgebras and differential Jordan supercoalgebra, and we obtain the classification of finite simple differential Jordan supercoalgebras up to isomorphism.

In \cite{KR} p.524, the notions of Jordan $H$-pseudosuperalgebra and Jordan conformal superalgebra were introduced. After some  computations, we observed that if we write the Jordan identity for the  $H$-pseudosuperalgebra in the special case of $H=\CC[\p]$, we find that the Jordan identity in the conformal case is not the Jordan identity presented in \cite{KR}, but rather, it requires some minor corrections that are presented in the following definition.

\ 

\begin{definition} A {\it   Jordan conformal superalgebra} $R$ is  a left
$\ZZ/2\ZZ$-graded $\cp$-module endowed with a $\CC$-linear map,
\begin{displaymath}
R\otimes R  \longrightarrow \CC[\la]\otimes R, \qquad a\otimes b
\mapsto a_\la b
\end{displaymath}

\vskip .1cm

\noindent called the $\la$-product, and  satisfying the following axioms
$(a,\, b,\, c\in R)$:

\

\noindent Conformal sesquilinearity: $ \qquad  \pa (a_\la b)=-\la
(a_\la b),\qquad a_\la \pa b=(\la+\pa) (a_\la b)$,

\vskip .3cm

\noindent Commutativity: $\ \qquad\qquad\qquad a_\la
b=(-1)^{|a| |b|} \, b_{-\la-\pa} \, a$,

\vskip .3cm

\noindent Jordan identity:
\begin{align*}
 & (-1)^{|a||c|}\  a_\la((b_\mu c)_\nu d)
+(-1)^{|a||b|}  \  b_\mu((c_{\nu-\mu}a)_{\la-\mu} d)
+(-1)^{|b||c|}  \  c_{\nu-\mu} ((a_{-\mu-\p}b)_{\la+\mu} d) \\
& =(-1)^{|a||c|}\  (a_{-\mu-\p}b)_{\la+\mu}(c_{\nu-\mu}d)
+(-1)^{|a||b|}  \  (b_\mu c)_\nu(a_{\la}d)
+(-1)^{|b||c|}  \  (c_{\nu-\mu}a)_{\la+\nu-\mu}(b_\mu d),
\end{align*}

\vskip .1cm

\noindent where  $|a|\in \ZZ/2\ZZ$ is the parity of $a$.
\end{definition}

A Jordan conformal superalgebra is called {\it finite} if it has finite
rank as a $\CC[\pa]$-module. The notions of homomorphism, ideal
and subalgebras of a Jordan conformal superalgebra are defined in the
usual way. A Jordan conformal superalgebra $R$ is {\it simple} if the $\la$-product is nontrivial  and it has no nontrivial proper ideals.

Some examples of Jordan conformal superalgebras are as follows: the Jordan conformal superalgebra $J_n=\CC[\p]\otimes (\La(n)\oplus \La(n) \, \theta)$ have the following $\la$-products (for $a,b\in \La(n)$ homogeneous elements):
\begin{align}\label{Jnn}
  a_\la b \, & = \;\! ab \nonumber\\
  a_\la b\;\!\theta & =(ab)\,\theta \nonumber\\
  a\;\!\theta_\la b & =(-1)^{|b|} (ab)\,\theta\\
a\;\!\theta_\la b\;\!\theta & =
(-1)^{|b|}
\Big[ \,\la \, (|a|+|b|-4)\, ab + (|a|-2)\, \p (ab)  \Big.\nonumber\\
& \qquad\qquad  + \Big. (-1)^{|a|}\Big(\dsum_{i=1}^{n-2} (\p_i a) (\p_i b) +
(\p_n a) (\p_{n-1} b) + (\p_{n-1} a) (\p_n b) \Big)\Big].\nonumber
\end{align}

\noindent Observe that the last $\la$-product in (\ref{Jnn})  already includes some minor corrections compared  to the one given in \cite{KR}, p.527. This can be easily proved by considering the Jordan formal distribution superalgebra $J(1,n)=\La(n)[t,t^{-1}]\oplus  \La(n)[t,t^{-1}]\theta$, with the formal distributions
\begin{equation*}
  a(z)=\sum_{n\in \ZZ} (a t^n)z^{-n-1}=a\, \de(t-z),\qquad \mathrm{and}\qquad a\theta (z)=\sum_{n\in \ZZ} (a \theta t^n)z^{-n-1}=a\theta \, \de(t-z),
\end{equation*}
associated to every $a\in\La(n)$, where $\de(t-z)=\sum_{n\in \ZZ}  t^n \, z^{-n-1}$ is the usual delta function. In this case, one has to compute the corresponding products of the formal distributions and then one has to apply the Fourier transform to obtain the corresponding $\la$-products  (see \cite{K} for details).

The Jordan conformal superalgebra $JS_1$ is freely spanned over $\cp$ by an even element $S$ and an odd element $T$ such that
\begin{equation*}
S_\la S=2S,\quad T_\la T=(2\la +\p)S,\quad T_\la S=T.
\end{equation*}

The Jordan conformal superalgebra $JCK_4$ is freely spanned over $\cp$ by even elements $1, \omega_1,\omega_2,\omega_3$ and  odd elements $x,x_1,x_2,x_3$, with the following $\la$-products (for any generator $A$)
\begin{align*}
  1_\la A & = A  \\
  \omega_i\,_{\la}\, \omega_i  & =
\left \{
    \begin{aligned}
      \ 1
&,\quad \   \text{if} \ \ \ i=1,2 \\
   -1
&,\quad \    \text{if} \   \  \  i=3
    \end{aligned}
  \right . \\
 \omega_i\,_{\la}\, \omega_j  & = \, 0\qquad\quad \  \mathrm{ if }\ \ \  i\neq j\\
\omega_i\,_{\la}\, x \  & = \, \la \, x_i \\
\omega_i\,_{\la}\, x_j   & =  \, - x_{i\times j} \\
x\,_{\la}\, x \  & = (2\la+\p) \, 1 \\
x_i\,_{\la}\, x \  & = \, \omega_i \\
x_i\,_{\la}\, x_j \,  & = \, 0\qquad \mathrm{for\  all }\ \ i,j,
\end{align*}
where $x_{1\times 2}=-x_{2\times 1} =x_3$, $x_{1\times 3}=-x_{3\times 1} =x_2$, $-x_{2\times 3}=x_{3\times 2} =x_1$, $x_{i\times i}=0$.

For  $\mathcal{J}$  a  Jordan superalgebra, the {\it current conformal superalgebra}
associated to $\mathcal{J}$ is defined by:
\begin{displaymath}
{\rm Cur }(\mathcal{J})=\cp \otimes \mathcal{J}, \qquad \quad [a_\la b]=[a,b],
\qquad a,b\in \mathcal{J}.
\end{displaymath}
\vskip .2cm

\

The following theorem is the main result of \cite{KR}.

\begin{theorem}\label{JJJ}
  A simple finite Jordan conformal superalgebra is isomorphic to one of the conformal superalgebras in the following list:

\vskip .1cm

1.  $J_n$;

2.  ${JS}_1$;

3.   ${JCK}_4$;

4.  a current conformal superalgebra over a simple finite-dimensional Jordan superalgebra.
\end{theorem}

\begin{definition} A {\it differential Jordan supercoalgebra} $(C,\Delta)$ is a $\ZZ/2\ZZ$-graded $\CC[\partial]$-module $C$ endowed with a $\CC[\partial]$-homomorphism
\begin{displaymath}
\De:C\to C\tt C
\end{displaymath}
such that it is co-commutative:
\begin{equation*}
  \tau \,\De=\De,
\end{equation*}
\noindent and satisfies the co-Jordan identity:
\vskip -.14cm
\begin{displaymath}
(1+\zeta+\zeta^2)(\De\otimes \Delta) \Delta = (1+\zeta+\zeta^2)(I\tt \De\tt I)(I\otimes \Delta) \Delta,
\end{displaymath}
\vskip .3cm

\noindent
where $\tau(a\otimes b )=(-1)^{|a| |b|}\, b\otimes a$, and $\zeta(a\tt b\tt c\tt d)=(-1)^{|a| (|b|+|c|)}\, b\tt c\tt a\tt d$.
\end{definition}

\noindent That is, the standard definition of Jordan supercoalgebra, with a compatible
$\CC[\partial]$-structure. A differential Jordan supercoalgebra is called {\it finite} if it has finite
rank as a $\CC[\pa]$-module. The notions of homomorphism 
and subsupercoalgebras of a differential Jordan supercoalgebra are defined in the
usual way.  A differential Jordan supercoalgebra $(C,\Delta)$ is {\it simple} if $\Delta\neq 0$ and contains no subsupercoalgebras except for zero and itself.

By a  simple computation and using the same proofs given in \cite{L} and in our Section 2, it is easy to show that the analog correspondence of categories in Theorem \ref{equival} holds for  finite free Jordan conformal algebras. Similarly, Proposition \ref{ideal}
and the construction given in Theorem \ref{prop:dual} also hold for  Jordan conformal (super)algebras and  differential Jordan (super)coalgebras.  Using these results, we have one of the main result of this work:

\begin{theorem}
Any  simple differential Jordan supercoalgebra of finite rank is isomorphic to one  of the following list:

\vskip .1cm

1. $(J_n)^{*_c}  $;

2. $({JS}_1)^{*_c} $;

3. $({JCK}_4)^{*_c}  $;

4. $(Cur(\mathfrak{g}))^{*_c}$, where $\mathfrak{g}$ is a simple finite-dimensional Jordan superalgebra.
\end{theorem}

\

Now, we present an explicit description of them. By a straightforward computation, we have:

\begin{theorem}
  The Jordan supercoalgebra structure on $(JS_1)^\c$ is defined by the following formulas, in the dual basis:
\begin{align*}
  \De(T^*) & = T^*\tt S^* + S^*\tt T^*, \\
  \De(S^*) & = 2\,S^*\tt S^* + \p T^*\tt T^*- T^*\tt \p T^*.
\end{align*}

\noindent  The Jordan supercoalgebra structure on $( JCK_4)^\c$ is defined by the following formulas, in the dual basis:
\begin{align*}
  \De(1^*) & = 1^*\tt 1^* + \omega_1^*\tt \om_1^*+\omega_2^*\tt \om_2^*-\omega_3^*\tt \om_3^*+\p x^*\tt x^*- x^*\tt \p x^*, \\
  \De(x^*) & = 1^*\tt x^* +  x^*\tt 1^*,\\
\De(\om_i^*) & = 1^*\tt \om_i^* + \omega_i^*\tt 1^*+ x_i^*\tt x^*-x^*\tt x_i^*,\\
\De(x_k^*) & = 1^*\tt x_k^* + x_k^*\tt 1^*+ \p\, \omega_k^*\tt x^*+x^*\tt \p\, \omega_k^*
-\sum_{\substack{i,j\in\{k\}^c\\ i\neq j}} (\omega_i^*\tt x_j^*+ x_j^*\tt \omega_i^*) .
\end{align*}

\end{theorem}

In the remaining case we have:

\begin{theorem}
  The Jordan supercoalgebra structure on $(J_n)^\c$ is defined by the following formulas, in the dual basis:

  \begin{align}\label{DJn1}
    \De ((\xi_K \theta)^*) &= \sum_{\substack{I,J \, : \, I\cap J=\emptyset \\
    \mathrm{ord}(I,J)=K}}  (-1)^{\alpha(I,J)}\, \Big[\, \xi_I^*\tt   (\xi_J \theta)^*+
    (-1)^{|I|\,(|J|+1)} \,
     (\xi_J \theta)^*\tt \xi_I^*\Big]
  \end{align}
and
\begin{align}\label{DJn2}
    \De &(\xi_K^*) = \sum_{\substack{I,J \, : \, I\cap J=\emptyset \\
    \mathrm{ord}(I,J)=K}}  (-1)^{\alpha(I,J)}\, \xi_I^*\tt \xi_J^* \nonumber\\
   & \ \ \ \quad +
\sum_{\substack{I,J \, : \, I\cap J=\emptyset \\
    \mathrm{ord}(I,J)=K}}
(-1)^{|J|+\alpha(I,J)}
   (|J|-2)
  \Big[ \p (\xi_I\,\theta)^*\tt (\xi_J\,\theta)^*
+ (-1)^{(|I|+1)(|J|+1)} (\xi_J\,\theta)^*\tt \p (\xi_I\,\theta)^*  \Big]
 \\
    & \ \ \ \quad +
\sum_{i=1}^{n-2}\sum_{\substack{I,J \, : \,i\in I, i\in J  \\ (I-\{i\})\cap (J-\{i\})=\emptyset \\
    \mathrm{ord}(I-\{i\},J-\{i\})=K}}
(-1)^{|I|+|J|+\varepsilon_i^I+\varepsilon_i^J+\, \alpha(I-\{i\},J-\{i\})}\, (\xi_I\,\theta)^*\tt (\xi_J\,\theta)^*
\nonumber\\
    & \ \ \ \quad +
\sum_{\substack{I,J \, : \,n-1\in I, n\in J \\ (I-\{n-1\})\cap (J-\{n\})=\emptyset \\
    \mathrm{ord}(I-\{n-1\},J-\{n\})=K}}
(-1)^{|I|+|J|+\varepsilon_{n-1}^I+\varepsilon_n^J+\, \alpha(I-\{n-1\},J-\{n\})}\, (\xi_I\,\theta)^*\tt (\xi_J\,\theta)^*
\nonumber\\
    & \ \ \ \quad +
\sum_{\substack{I,J \, : \,n\in I, n-1\in J \\ (I-\{n\})\cap (J-\{n-1\})=\emptyset \\
    \mathrm{ord}(I-\{n\},J-\{n-1\})=K}}
(-1)^{|I|+|J|+\varepsilon_{n}^I+\varepsilon_{n-1}^J+\, \alpha(I-\{n\},J-\{n-1\})}\, (\xi_I\,\theta)^*\tt (\xi_J\,\theta)^*.\nonumber
  \end{align}
\end{theorem}

\vskip .4cm

\begin{proof}
\noindent Using (\ref{eq:De}), we have
\begin{align*}
    \de ((\xi_K \theta)^*))
     &=\sum_{I,J}\, (\xi_K \theta)^*_\mu \big([{\xi_I}_\la \xi_J\,\theta]\big) (\xi_I^*\tt (\xi_J\,\theta)^*) + \sum_{I,J} \, (\xi_K \theta)^*_\mu \big([{\xi_J\,\theta}_\la \xi_I]\big) ((\xi_J\,\theta)^*\tt \xi_I^*),
\end{align*}
\noindent where we must remember that we have to replace $
\la=\p\otimes 1$ and $\mu= -\p\otimes 1 - 1\otimes \p $. Then
\vskip .2cm
\begin{align*}
\de ((\xi_K \theta)^*))
 &=\sum_{I,J}\, (\xi_K \theta)^*_\mu \big({\xi_I} \xi_J\,\theta\big) (\xi_I^*\tt (\xi_J\,\theta)^*) + \sum_{I,J} \, (\xi_K \theta)^*_\mu \big((-1)^{|I|}{\xi_J} \xi_I\,\theta\big) ((\xi_J\,\theta)^*\tt \xi_I^*)\\
 &= \sum_{\substack{I,J \, : \, I\cap J=\emptyset \\
    \mathrm{ord}(I,J)=K}}  (-1)^{\alpha(I,J)}\, \xi_I^*\tt (\xi_J \, \theta)^*
     +   \sum_{\substack{I,J \, : \, I\cap J=\emptyset \\
    \mathrm{ord}(I,J)=K}}  (-1)^{|I|+\alpha(J,I)}\, (\xi_J\,\theta)^*\tt \xi_I^*\\
     &= \sum_{\substack{I,J \, : \, I\cap J=\emptyset \\
    \mathrm{ord}(I,J)=K}}  (-1)^{\alpha(I,J)}\, \Big[\, \xi_I^*\tt   (\xi_J \theta)^*+
    (-1)^{|I|\,(|J|+1)} \,
     (\xi_J \theta)^*\tt \xi_I^*\Big]
          ,
     \nonumber
  \end{align*}
obtaining (\ref{DJn1}). Similarly, we have
\begin{align}\label{J100}
    \de (\xi_K^*)
     &=\sum_{I,J}\, (\xi_K^*)_\mu \big([{\xi_I}_\la \xi_J]\big) (\xi_I^*\tt \xi_J^*) + \sum_{I,J} \, (\xi_K^*)_\mu \big([{(\xi_I\,\theta)}_\la (\xi_J\,\theta)]\big) ((\xi_I\,\theta)^*\tt (\xi_J\,\theta)^*),
\end{align}
\noindent where we must remember that we have to replace $
\la=\p\otimes 1$ and $\mu= -\p\otimes 1 - 1\otimes \p $. Then, it is clear that the first sum in (\ref{J100}) is equal to
\begin{equation*}
  \sum_{\substack{I,J \, : \, I\cap J=\emptyset \\
    \mathrm{ord}(I,J)=K}}  (-1)^{\alpha(I,J)}\, \xi_I^*\tt \xi_J^*
\end{equation*}
which is the first term in (\ref{DJn2}). The second sum in (\ref{J100}) is equal to
\begin{align}\label{termino}
  & \sum_{I,J}  (-1)^{|J|} (\xi_K^*)_\mu
\Bigg(\,\la\, (|I|+|J|-4)\,\xi_I\xi_J +
(|I|-2)\,\p\, \xi_I\xi_J \nonumber
 \\
&\qquad \ \ \ \ + (-1)^{|I|} \Big[\sum_{i=1}^{n-2}  (\p_i \xi_i) (\p_i \xi_J) +
(\p_{n-1} \xi_I) (\p_n \xi_J)+ (\p_n \xi_I) (\p_{n-1} \xi_J)\Big]
\Bigg)((\xi_I\,\theta)^*\tt (\xi_J\,\theta)^*)\nonumber \\
    & \  = \sum_{\substack{I,J \, : \, I\cap J=\emptyset \\
    \mathrm{ord}(I,J)=K}}
(-1)^{|J|} \Big[\,\la\, (|I|+|J|-4)\,
             + (|I|-2)\,\mu\, \Big]\, (-1)^{\alpha(I,J)} \, (\xi_I\,\theta)^*\tt (\xi_J\,\theta)^*\nonumber\\
    & \ \ \ \quad +
\sum_{i=1}^{n-2}\sum_{\substack{I,J \, : \,i\in I, i\in J  \nonumber\\ (I-\{i\})\cap (J-\{i\})=\emptyset \\
    \mathrm{ord}(I-\{i\},J-\{i\})=K}}
(-1)^{|I|+|J|+\varepsilon_i^I+\varepsilon_i^J+\, \alpha(I-\{i\},J-\{i\})}\, (\xi_I\,\theta)^*\tt (\xi_J\,\theta)^*\\
    & \ \ \ \quad +
\sum_{\substack{I,J \, : \,n-1\in I, n\in J \\ (I-\{n-1\})\cap (J-\{n\})=\emptyset \\
    \mathrm{ord}(I-\{n-1\},J-\{n\})=K}}
(-1)^{|I|+|J|+\varepsilon_{n-1}^I+\varepsilon_n^J+\, \alpha(I-\{n-1\},J-\{n\})}\, (\xi_I\,\theta)^*\tt (\xi_J\,\theta)^*\\
    & \ \ \ \quad +
\sum_{\substack{I,J \, : \,n\in I, n-1\in J \\ (I-\{n\})\cap (J-\{n-1\})=\emptyset \\
    \mathrm{ord}(I-\{n\},J-\{n-1\})=K}}
(-1)^{|I|+|J|+\varepsilon_{n}^I+\varepsilon_{n-1}^J+\, \alpha(I-\{n\},J-\{n-1\})}\, (\xi_I\,\theta)^*\tt (\xi_J\,\theta)^*.
\nonumber
\end{align}
Observe that  the last three terms in (\ref{termino}) correspond to the last three terms in (\ref{DJn2}). We still have to replace  $
\la=\p\otimes 1$ and $\mu= -\p\otimes 1 - 1\otimes \p $ in the first term in (\ref{termino}). Note that
\begin{align*}
    & (-1)^{|J|} \Big[\,  (|I|+|J|-4)\, (\p\otimes 1)
             + (|I|-2)\,(-\p\otimes 1 - 1\otimes \p) \, \Big]\, (-1)^{\alpha(I,J)} \\
   & \quad = (-1)^{|J|+\alpha(I,J)}
\Big[\,  (|J|-2)\, (\p\otimes 1)
             - (|I|-2)\,(  1\otimes \p) \, \Big].
\end{align*}
Therefore, the first term in (\ref{termino}) is equal to
\begin{align*}
  & \sum_{\substack{I,J \, : \, I\cap J=\emptyset \\
    \mathrm{ord}(I,J)=K}}
(-1)^{|J|+\alpha(I,J)}
\Big[\,  (|J|-2)\, (\p\otimes 1)
             - (|I|-2)\,(  1\otimes \p) \, \Big] \, (\xi_I\,\theta)^*\tt (\xi_J\,\theta)^*
 \\
&  \qquad =
\sum_{\substack{I,J \, : \, I\cap J=\emptyset \\
    \mathrm{ord}(I,J)=K}}
(-1)^{|J|+\alpha(I,J)}
   (|J|-2)
 \, \p (\xi_I\,\theta)^*\tt (\xi_J\,\theta)^*\\
& \hskip 6cm - \sum_{\substack{I,J \, : \, I\cap J=\emptyset \\
    \mathrm{ord}(J,I)=K}}
(-1)^{|I|+\alpha(J,I)}
   (|J|-2)
 \, (\xi_J\,\theta)^*\tt \p (\xi_I\,\theta)^*
\end{align*}
and, using (\ref{alfa}), it can be rewritten as
\begin{equation*}
  \sum_{\substack{I,J \, : \, I\cap J=\emptyset \\
    \mathrm{ord}(I,J)=K}}
(-1)^{|J|+\alpha(I,J)}
   (|J|-2)
 \, \Big[\, \p (\xi_I\,\theta)^*\tt (\xi_J\,\theta)^*
+ (-1)^{(|I|+1)(|J|+1)} (\xi_J\,\theta)^*\tt \p (\xi_I\,\theta)^* \, \Big],
\end{equation*}
obtaining the second term in (\ref{DJn2}).
\end{proof}

\

\noindent{\bf Acknowledgment.} C. Boyallian and J. Liberati were
supported in part by grants of  Conicet and  Secyt-UNC  (Argentina).



\


\begin{thebibliography}{11111}




\bibitem [BDK]{BDK} B. Bakalov, A. D'Andrea and V. G. Kac, {\em Theory of finite pseudoalgebras.} Adv. Math. {\bf 162} (2001), no. 1,
1--140.

\bibitem [BKLR]{BKLR}Carina Boyallian, Victor G. Kac, Jose I. Liberati, Alexei Rudakov; {\em Representations of simple finite Lie conformal superalgebras of type $W$ 
 and $S$.} J. Math. Phys. 1 April 2006; {\bf 47} (4): 043513. https://doi.org/10.1063/1.2191788

\bibitem [CK]{CK}  S. Cheng  and  V. G. Kac, {\em Conformal
    modules}, Asian J. Math. {\bf 1} (1997), 181-193.  Erratum: 2
  (1998), 153-156.


\bibitem [CK2]{CK2}  S. Cheng  and  V. G. Kac, {\em  A
new $N=6$ superconformal algebra}, Comm. Math. Phys. {\bf 186}
(1997), no. 1, 219-231.



\bibitem[CL]{CL}
S. Cheng  and N. Lam, {\textit{Finite conformal modules over
$N=2,3,4$ superconformal algebras.}}, Journal of Math. Phys. {\bf
42}, No.2, 906-933 (2001).




\bibitem [DK]{DK} A. D'Andrea and    V. G. Kac, {\em
Structure theory of finite conformal algebras}, Selecta Math.
(n.S.) {\bf 4} (1998), 377-418.


\bibitem [FK]{FK}  D. Fattori and V. G. Kac,  {\em Classification of
finite simple Lie conformal superalgebras.} J. Algebra {\bf 258}
(2002), no. 1, 23--59.


\bibitem[K]{K} V.~G.~Kac,
{\textit{Vertex algebras for beginners}}, University Lecture
Series, 10. American Mathematical Society, Providence, RI, 1996.
Second edition 1998.

\bibitem[KR]{KR} V. G. Kac and  A. Retakh, {\em Simple Jordan conformal superalgebras.} J. Algebra Appl. {\bf{7}} (2008), no. 4, 517--533.

\bibitem[L]{L} Jose I. Liberati, {\em On conformal bialgebras.} J. Algebra {\bf 319} (2008), no. 6, 2295--2318.




\end{thebibliography}
\end{document}